\documentclass{article}[12pt]
\usepackage{amsmath, amsfonts, amscd, amssymb, amsthm, enumerate, xypic, psfrag}

\DeclareMathOperator{\WS}{\operatorname{WS}}
\DeclareMathOperator{\WA}{\operatorname{WA}}
\DeclareMathOperator{\Mat}{\operatorname{M}}
\DeclareMathOperator{\Mats}{\operatorname{S}}
\DeclareMathOperator{\Mata}{\operatorname{A}}
\DeclareMathOperator{\Diag}{\operatorname{Diag}}
\DeclareMathOperator{\GL}{\operatorname{GL}}

\DeclareMathOperator{\Ker}{\operatorname{Ker}}
\DeclareMathOperator{\Vect}{\operatorname{span}}
\DeclareMathOperator{\im}{\operatorname{Im}}

\DeclareMathOperator{\rk}{\operatorname{rk}}
\DeclareMathOperator{\urk}{\operatorname{urk}}

\renewcommand{\setminus}{\smallsetminus}

\def\K{\mathbb{K}}

\def\calS{\mathcal{S}}
\def\calT{\mathcal{T}}

\def\calV{\mathcal{V}}
\def\calW{\mathcal{W}}

\def\lcro{\mathopen{[\![}}
\def\rcro{\mathclose{]\!]}}

\theoremstyle{definition}
\newtheorem{Def}{Definition}[section]
\newtheorem{Not}[Def]{Notation}

\theoremstyle{plain}
\newtheorem{theo}{Theorem}[section]
\newtheorem{prop}[theo]{Proposition}
\newtheorem{cor}[theo]{Corollary}
\newtheorem{lemme}[theo]{Lemma}
\newtheorem{claim}{Claim}

\theoremstyle{plain}

\theoremstyle{remark}

\title{Large spaces of symmetric or alternating matrices with bounded rank}
\author{Cl\'ement de Seguins Pazzis\footnote{Universit\'e de Versailles Saint-Quentin-en-Yvelines, Laboratoire de Math\'ematiques
de Versailles, 45 avenue des Etats-Unis, 78035 Versailles cedex, France}
\footnote{e-mail address: dsp.prof@gmail.com}}

\begin{document}

\thispagestyle{plain}

\maketitle

\begin{abstract}
Let $r$ and $n$ be positive integers such that $r<n$, and $\K$ be an arbitrary field.
In a recent work, we have determined the maximal dimension for a linear subspace of $n$ by $n$ symmetric matrices
with rank less than or equal to $r$, and we have classified the spaces having that maximal dimension.
In this article, provided that $\K$ has more than two elements, we extend this classification to spaces
whose dimension is close to the maximal one: this generalizes a result of Loewy \cite{Loewy}.
We also prove a similar result on spaces of alternating matrices with bounded rank, with no restriction on the cardinality
of the underlying field.
\end{abstract}



\section{Introduction}

\subsection{The problem}

Let $\K$ be a (commutative) field. We denote:
\begin{itemize}
\item By $\Mat_{n,p}(\K)$ the space of all $n$ by $p$ matrices with entries in $\K$, and
$\Mat_n(\K)$ is defined as $\Mat_{n,n}(\K)$;
\item By $\Mats_n(\K)$ the space of all $n$ by $n$ symmetric matrices with entries in $\K$;
\item By $\Mata_n(\K)$ the space of all $n$ by $n$ alternating matrices with entries in $\K$
(that is, the skew-symmetric matrices with diagonal zero or, equivalently, the matrices $A \in \Mat_n(\K)$
such that $X^T AX=0$ for all $X \in \K^n$);
\item By $\GL_n(\K)$ the group of all invertible matrices of $\Mat_n(\K)$.
\end{itemize}
Given integers $i \in \lcro 1,n\rcro$ and $j \in \lcro 1,p\rcro$, we denote by $E_{i,j}$
the matrix of $\Mat_{n,p}(\K)$ with zero entries everywhere except at the $(i,j)$ spot, where the entry equals $1$.

Two subsets $\calV$ and $\calW$ of $\Mat_n(\K)$ are called \textbf{congruent} when
there exists a matrix $P \in \GL_n(\K)$ such that
$$\calV=P\,\calW\, P^T,$$
i.e.\ $\calV$ and $\calW$ represent the same set of bilinear forms in a different choice of basis of $\K^n$.

The \textbf{upper-rank} of a non-empty subset $\calV$ of $\Mat_{n,p}(\K)$ is defined as
$$\urk \calV:=\max \bigl\{\rk M \mid M \in \calV\bigr\}.$$

Spaces of matrices with bounded rank have attracted much focus from the mathematical community in the
last fifty years. In the case of rectangular matrices, the basic theorem is the one of Flanders \cite{Flanders},
which states that if a linear subspace $S$ of $\Mat_{n,p}(\K)$ has upper-rank less than or equal to $r$, with
$r \leq p \leq n$, then $\dim S \leq nr$. Moreover, if $\dim S=nr$ then either there exists a $(p-r)$-dimensional linear subspace of
$\K^p$ on which all the matrices of $S$ vanish, or else $n=p$ and there exists an $r$-dimensional linear subspace of $\K^n$
that includes the range of every matrix of $S$. By perfecting Flanders's techniques, Atkinson, Lloyd and Beasley \cite{AtkLloyd,Beasley} later proved
that the latter statement actually holds whenever $\dim S>nr-(n-p+r)+1$, provided that the underlying field
be of cardinality greater than $r$. Recently, we have been able to extend this result to arbitrary fields \cite{dSPboundedrank},
and we have even managed to obtain a complete classification of the situation when $\dim S\geq nr-2(n-p+r)+2$ and $S$ is a maximal
space with upper-rank less than or equal to $r$, unless $\K$ has two elements only
\cite{dSPboundedrankv2}. In short, maximal spaces with upper-rank at most $r$ and dimension close to the critical one $nr$ are known.

In this article, we shall be concerned with the corresponding problem for subspaces of symmetric or alternating matrices.
Results on this topic have been obtained much later than in the rectangular case, and they have a reputation of being much more difficult.
The first results were obtained by Meshulam \cite{Meshulamsymmetric}, who determined the maximal dimension for a linear subspace of
$\Mats_n(\K)$ (or $\Mata_n(\K)$) with upper-rank at most $r$, provided that the field $\K$ does not have characteristic $2$
and its cardinality is large with respect to $n$ and $r$. Later, by following Meshulam's core ideas,
Loewy and Radwan \cite{LoewyRadwan} were able to characterize the subspaces of symmetric matrices having the critical dimension.
In a recent work \cite{dSPsym}, we have used new methods to
generalize the results of Meshulam, Loewy and Radwan to an arbitrary field (with arbitrary characteristic and cardinality), even
allowing affine subspaces instead of linear subspaces only. In this article, our ambition is
to extend this result to a large range of dimensions below the critical one: in short, our results will be the equivalent of the one of
Atkinson, Lloyd and Beasley in the theory of large spaces of bounded-rank rectangular matrices.
Before we go on, we should point out that a special case of our theorems has already been obtained by Loewy in \cite{Loewy}:
we will discuss the limitations of his result later on.

\vskip 3mm
It is high time we gave examples of large maximal subspaces of $\Mata_n(\K)$ and $\Mats_n(\K)$ with upper-rank $r$.
The following special spaces are the equivalent of the so-called ``compression spaces" in the theory of large spaces of bounded-rank rectangular matrices.
Let $n$, $s$ and $t$ be non-negative integers such that $2s+t \leq n$.
We define $\WS_{n,s,t}(\K)$ (respectively $\WA_{n,s,t}(\K)$) as the linear subspace of all matrices of $\Mats_n(\K)$
(respectively, of $\Mata_n(\K)$) of the form
$$\begin{bmatrix}
[?]_{s \times s} & [?]_{s \times t} & [?]_{s \times (n-s-t)} \\
[?]_{t \times s} & [?]_{t \times t} & [0]_{t \times (n-s-t)} \\
[?]_{(n-s-t) \times s} & [0]_{(n-s-t) \times t} & [0]_{(n-s-t) \times (n-s-t)}
\end{bmatrix}.$$
One checks that
$$\urk \WS_{n,s,t}(\K)=2s+t$$
and, if $t$ is odd, that
$$\urk \WA_{n,s,t}(\K)=2s+t-1.$$
Setting
$$s_{n,s,t}:=\dim \WS_{n,s,t}(\K) \quad \text{and} \quad a_{n,s,t}:=\dim \WA_{n,s,t}(\K),$$
one computes that
$$s_{n,s,t}=\dbinom{s+1}{2}+\dbinom{t+1}{2}+s(n-s)$$
and
$$a_{n,s,t}=\dbinom{s}{2}+\dbinom{t}{2}+s(n-s).$$
It is easily checked that $\WS_{n,s,t}(\K)$ is a maximal linear subspace of $\Mats_n(\K)$ with upper-rank $2s+t$, and that
if $t$ is odd $\WA_{n,s,t}(\K)$ is a maximal linear subspace of $\Mata_n(\K)$ with upper-rank $2s+t-1$.

Now, let $n$ and $r$ be integers such that $0 \leq r \leq n$.
One checks that both sequences $(s_{n,s,r-2s})_{0 \leq s \leq \lfloor r/2\rfloor}$
and $(a_{n,s,r+1-2s})_{0 \leq s \leq \lfloor (r+1)/2\rfloor}$ are strictly convex.
Indeed, both functions $s \mapsto \frac{(s+1)s}{2}+\frac{(r-2s+1)(r-2s)}{2}+s(n-s)$
and $s \mapsto \frac{s(s-1)}{2}+\frac{(r-2s+1)(r-2s)}{2}+s(n-s)$
are polynomials of degree $2$ and with coefficient on $s^2$ equal to $\frac{3}{2}\cdot$

We can now recall the following results of \cite{dSPsym}, which extend those of Loewy, Meshulam and Radwan
to arbitrary fields:

\begin{theo}\label{oldtheoalt}
Let $n$ and $s$ be non-negative integers such that $2s<n$.
Let $S$ be a linear subspace of $\Mata_n(\K)$ such that $\urk S \leq 2s$.
Then,
$$\dim S \leq \max\bigl(a_{n,0,2s+1},a_{n,s,1}\bigr),$$
and if equality occurs then $S$ is congruent to $\WA_{n,s,1}(\K)$ or to $\WA_{n,0,2s+1}(\K)$.
\end{theo}

\begin{theo}\label{oldtheosym}
Let $n$ and $r$ be non-negative integers such that $r<n$.
Let $S$ be a linear subspace of $\Mats_n(\K)$ such that $\urk S \leq r$.
\begin{enumerate}[(a)]
\item If $r=2s$ for some integer $s$, then
$$\dim S \leq \max\bigl(s_{n,0,r},s_{n,s,0}\bigr),$$
and if equality occurs then either $S$ is congruent to $\WS_{n,s,0}(\K)$ or to $\WS_{n,0,r}(\K)$,
or $\K$ has characteristic $2$ and $S$ is congruent to $\WA_{n,0,r+1}(\K)$. \\
\item If $r=2s+1$ for some integer $s$, then
$$\dim S \leq \max\bigl(s_{n,0,r},s_{n,s,1}\bigr),$$
and if equality occurs then  $S$ is congruent to $\WS_{n,s,1}(\K)$ or to $\WS_{n,0,r}(\K)$.
\end{enumerate}
\end{theo}

Beware that our notation in this paper is different from the one in \cite{dSPsym}.
Here is the full list of correspondence between the two sets of notation.

\begin{center}
\begin{tabular}{| c | c |}
\hline
Notation in this article & Notation in \cite{dSPsym} \\
\hline
\hline
$a_{n,0,2s+1}$ & $a_{n,2s}^{(1)}$ \\
\hline
$a_{n,s,1}$ & $a_{n,2s}^{(2)}$ \\
\hline
$\WA_{n,0,2s+1}(\K)$ & $\widetilde{\Mata_{2s+1}(\K)}^{(n)}$ \\
\hline
$\WA_{n,s,1}(\K)$ & $\WA_{n,s}(\K)$ \\
\hline
$s_{n,0,r}$ & $s_{n,r}^{(1)}$ \\
\hline
$s_{n,s,0}$ & $s_{n,2s}^{(2)}$ \\
\hline
$s_{n,s,1}$ & $s_{n,2s+1}^{(2)}$ \\
\hline
$\WS_{n,0,r}(\K)$ & $\widetilde{\Mats_r(\K)}^{(n)}$ \\
\hline
$\WS_{n,s,0}(\K)$ & $\WS_{n,2s}(\K)$ \\
\hline
$\WS_{n,s,1}(\K)$ & $\WS_{n,2s+1}(\K)$ \\
\hline
\end{tabular}

\end{center}

\subsection{Main results}

We are ready to state our main theorem on large spaces of bounded-rank symmetric matrices.

\begin{theo}\label{maintheosym}
Assume that $\# \K>2$.
Let $S$ be a linear subspace of $\Mats_n(\K)$ and $r$ be an integer such that $1 < r < n$.
Write $r=2s+\epsilon$ for some positive integer $s$ and some $\epsilon \in \{0,1\}$.
Assume that $\K$ has characteristic not $2$ and that
$$\dim S >\max\left(s_{n,1,r-2},s_{n,s-1,2+\epsilon}\right) \quad \text{and} \quad \urk S \leq r.$$
Then, one of the following outcomes holds:
\begin{enumerate}[(i)]
\item $S$ is congruent to a subspace of $\WS_{n,0,r}(\K)$;
\item $S$ is congruent to a subspace of $\WS_{n,s,\epsilon}(\K)$;
\item $\K$ has characteristic $2$, $r$ is even and $S$ is congruent to a subspace of $\WA_{n,0,r+1}(\K)$.
\end{enumerate}
\end{theo}

The case of the field with two elements appears to be much more complicated, and we will not tackle it here.
We are confident that by using a similar strategy good results can be obtained on this difficult case.

Now, here is our result for spaces of alternating matrices. This time around, we do not exclude fields with two elements.

\begin{theo}\label{maintheoalt}
Let $S$ be a linear subspace of $\Mata_n(\K)$, and $r=2s$ be an even integer such that $0<r<n$.
Assume that
$$\dim S >\max\left(a_{n,1,r-1},a_{n,s-1,3}\right) \quad \text{and} \quad \urk S \leq r.$$
Then, $S$ is congruent to a subspace of $\WA_{n,0,r+1}(\K)$ or to a subspace of $\WA_{n,s,1}(\K)$.
\end{theo}

In the above two theorems, we cannot have $s=1$ for the assumption on the dimension of $S$ would
lead to a contradiction with Theorems \ref{oldtheoalt} and \ref{oldtheosym}. Understanding this special case will however be necessary
in the proofs of these theorems, and hence we shall include a discussion of it.

\begin{prop}\label{altrank2}
Let $S$ be a linear subspace of $\Mata_n(\K)$ such that $\dim S>3$ and $\urk S \leq 2$.
Then, $S$ is congruent to a subspace of $\WA_{n,1,1}(\K)$.
\end{prop}

The next result is already known \cite{ChooLimNg,Lim}:

\begin{prop}\label{symrank2}
Assume that $\# \K>2$.
Let $S$ be a linear subspace of $\Mats_n(\K)$ such that $\dim S>3$ and $\urk S \leq 2$.
Then, $S$ is congruent to a subspace of $\WS_{n,1,0}(\K)$.
\end{prop}

The following one seems to be new:

\begin{prop}\label{symrank3}
Assume that $\# \K>2$.
Let $S$ be a linear subspace of $\Mats_n(\K)$ such that $\dim S>6$ and $\urk S \leq 3$.
Then, $S$ is congruent to a subspace of $\WS_{n,1,1}(\K)$.
\end{prop}

Now, let us compare our results with Loewy's. Here is a restatement of the latter into our notation system:

\begin{theo}[Loewy \cite{Loewy}]
Assume that the characteristic of $\K$ is not $2$.
Let $S$ be a linear subspace of $\Mats_n(\K)$, and $s$ be a positive integer such that $2s<n$ and $2n>5s+1$.
Assume that
$$\urk S \leq 2s, \quad \# \K \geq 2s+3 \quad \text{and} \quad \dim S > \max(s_{n,s-1,2},s_{n,0,2s}).$$
Then, $S$ is congruent to a subspace of $\WS_{n,s,0}(\K)$.
\end{theo}

Unless $s=1$, the condition $\dim S > \max(s_{n,s-1,2},s_{n,0,2s})$ implies $\dim S > \max(s_{n,s-1,2},s_{n,1,2s-2})$
since $(s_{n,k,2s-2k})_{0 \leq k \leq s}$ is a convex sequence. For low values of $s$ with respect to $n$,
one sees that the two conditions are actually equivalent. If $s=1$, then Loewy's result is
a consequence of Proposition \ref{symrank2}.

Thus, Loewy's result is essentially a version of Theorem \ref{maintheosym}, for low even upper-ranks only, and
with some unnecessary assumptions on the field. Unsurprisingly, our methods will be very different from Loewy's.
Following Meshulam, Loewy's strategy consisted in exploiting a connection with graph theory.
On the contrary, our methods are firmly rooted in linear algebra: They can be seen as an adaptation of
techniques that have helped advance the theory of large spaces of bounded-rank rectangular matrices in the recent past
(see \cite{dSPboundedrankv2}).

\subsection{Main strategy}

Let us give a few ideas of the strategy of proof of Theorems \ref{maintheosym} and \ref{maintheoalt}.
We will limit our discussion to the alternating case, since it is the one that involves the lesser amount of technicalities.

We shall follow the main ideas of \cite{dSPsym}. Let $S$ be a linear subspace of $\Mata_n(\K)$ with upper-rank
less than or equal to $r=2s$ for some even integer $r<n-1$. We wish to prove that if the dimension of $S$
is large enough, then $S$ is congruent to a subspace of $\WA_{n,s,1}(\K)$ or to a subspace of $\WA_{n,0,r+1}(\K)$.

Our starting point is a simple observation.
Throughout the article, we shall write every matrix $M$ of $\Mata_n(\K)$  as
$$M=\begin{bmatrix}
0 & [?]_{1 \times (n-2)} & ? \\
[?]_{(n-2) \times 1} & K(M) & [?]_{(n-2) \times 1} \\
? & [?]_{1 \times (n-2)} & 0
\end{bmatrix}=\begin{bmatrix}
P(M) & [?]_{(n-1) \times 1} \\
[?]_{1 \times (n-1)} & 0
\end{bmatrix}
$$
with $K(M) \in \Mata_{n-2}(\K)$ and $P(M) \in \Mata_{n-1}(\K)$.
Assume that the matrix $N:=E_{1,n}-E_{n,1}$ belongs to $S$. Then, for all $M \in S$,
by using the fact that $\rk(M+tN) \leq r$ for all $t$ in $\K$, it can be shown (see Lemma 2.4 of \cite{dSPsym})
that $\rk K(M) \leq r-2$. Hence, $K(S)$ is a linear subspace of $\Mata_{n-2}(\K)$ with upper-rank at most $r-2$.
This should allow us to perform an inductive proof!
Yet, to perform such a proof it is necessary to know that the dimension of $K(S)$ is large enough.
This can be done thanks to the rank theorem: denoting by $W$ the linear subspace of $S$ consisting of all its matrices $M$ such that $P(M)=0$,
and by $W'$ the linear subspace of $P(S)$ consisting of all its matrices with first row zero, we find that
$$\dim K(S)=\dim S-\dim W-\dim W'.$$
Hence, if we can ensure that $\dim W$ is small enough, then we will know that $\dim K(S)$ is large enough, so that we can use the induction hypothesis.
Now, say that this has been obtained \emph{and} that $K(S)$ is congruent to a subspace of $\WA_{n-2,s-1,1}(\K)$
(it could also happen that $K(S)$ be congruent to a subspace of $\WA_{n-2,0,r-1}(\K)$, but let us not get distracted by this side issue).
Then, we can simply assume that $K(S) \subset \WA_{n-2,s-1,1}(\K)$, whence $P(S) \subset \WA_{n-1,s,1}(\K)$.
From there, we wish to prove that $S$ is congruent to a subspace of $\WA_{n,s,1}(\K)$; in rough terms, we want to \emph{lift}
the structure of $P(S)$ to the one of $S$.
The basic argument is the following one: every matrix of $S$ now splits as
$$M=\begin{bmatrix}
[?]_{s \times s} & -B(M)^T & [?]_{s \times 1} \\
B(M) & [0]_{(n-s-1) \times (n-s-1)} & C(M) \\
[?]_{1 \times s} & -C(M)^T & 0
\end{bmatrix}$$
with $B(M) \in \Mat_{n-s-1,s}(\K)$ and $C(M) \in \K^{n-s-1}$.
Then, for every $M$ in $S$, we see that
$$\rk M \geq \rk A(M)+\rk(-A(M)^T)$$
where
$$A(M):=\begin{bmatrix}
B(M) & C(M)
\end{bmatrix} \in \Mat_{n-s-1,s+1}(\K).$$
Hence, $\urk A(S) \leq s$. As $n-s-1 \geq s+1$, if the dimension of $P(S)$ is large enough then
so is the dimension of $A(S)$, and then we can use the classification of large spaces of rectangular matrices with bounded rank, and more specifically
the theorem of Atkinson and Lloyd \cite{AtkLloyd} that was recently generalized to all fields \cite{dSPboundedrank}.
If the dimension of $B(S)$ is large enough, it will yield that, for some fixed vector $X \in \K^s$, we have $C(M)=B(M)X$ for all $M \in S$;
from that point it will be easy to conclude that $S$ is congruent to a subspace of $\WA_{n,s,1}(\K)$.

Now, let us come back to the general situation. We introduce the following notation:

\begin{Not}
Let $H$ be a linear hyperplane of $\K^n$ and $V$ be a subset of $\Mat_n(\K)$.
We denote by $V_H$ the set of all matrices $M \in V$ for which $H$ is totally singular,
that is
$$\forall (X,Y) \in H^2, \; X^T MY=0.$$
\end{Not}

For example, if $H=\K^{n-1} \times \{0\}$ then $V_H$ consists of all the matrices $M$ of $V$
that have the following shape
$$M=\begin{bmatrix}
[0]_{(n-1) \times (n-1)} & [?]_{(n-1) \times 1} \\
[?]_{1 \times (n-1)} & ?
\end{bmatrix};$$
hence, in the above notation, we have $W=S_H$.
Therefore, we can use the above technique provided that $S_H$ be non-zero (in the above setting, it has to contain $E_{1,n}-E_{n,1}$)
and yet with small dimension.

In general, it is known (see Lemma 2.10 from \cite{dSPsym}) that there exists a linear hyperplane $H$ of $\K^n$ such that
$\dim S_H \leq s$. Here, we shall need a more precise result: if $\dim S_H \geq s$ for every linear hyperplane $H$ of $\K^n$
then $S$ is congruent to $\WA_{n,s,1}(\K)$ (see Proposition \ref{keylemmaalternate}).
Hence, either our conclusion is satisfied or we can find a linear hyperplane $H$ of $\K^n$ such that $\dim S_H \leq s-1$.

It remains to explain what to do if we have a linear hyperplane $H$ of $\K^n$ such that $S_H=\{0\}$.
Assume for instance that this is the case with $H=\K^{n-1} \times \{0\}$.
Then, in the above notation we apply induction to $P(S)$, which has the same dimension as $S$.
It is obvious that the requirements of Theorem \ref{maintheoalt} are fulfilled by $P(S)$,
and hence by induction no generality is lost in assuming that $P(S) \subset \WA_{n-1,s,1}(\K)$ or
$P(S) \subset \WA_{n,0,r+1}(\K)$. In the first case, we can use the same lifting technique as in the above to recover that $S$
is congruent to a subspace of $\WA_{n,s,1}(\K)$. To finish the discussion, assume that the second case holds
and let us see how to lift that information on $P(S)$ to obtain that $S$ is congruent to a subspace of $\WA_{n,0,r+1}(\K)$.
We have a subspace $V$ of $\Mata_{r+1}(\K)$, with $\dim V=\dim S$, and we have linear maps $C_1 : V \rightarrow \K^{r+1}$
and $C_2 : V \rightarrow \K^{n-r-2}$ such that $S$ is the space of all matrices
$$\begin{bmatrix}
A & [0]_{(r+1) \times (n-r-2)} & C_1(A) \\
[0]_{(n-r-2)\times (r+1)} & [0]_{(n-r-2) \times (n-r-2)} & C_2(A) \\
-C_1(A)^T & -C_2(A)^T & 0
\end{bmatrix}\quad \text{with $A \in V$.}$$
Moreover, the codimension of $V$ in $\Mata_{r+1}(\K)$ is small.
Then, by using results on affine spaces of bounded-rank alternating matrices (Corollary 3.1 from \cite{dSPsym}),
we will prove that $C_2=0$ and that $C_1(A)$ belongs to the column space of $A$ for all $A$ in $V$.
In other words, $C_1$ is a \emph{range-compatible} map. Thanks to recent new results on such maps (see \cite{dSPRCsym}),
we shall conclude that $C_1 : A \mapsto AX$ for some fixed $X \in \K^{r+1}$, and hence $S$ is congruent to $\WA_{n,0,r+1}(\K)$.

\subsection{Structure of the article}

The article is laid out as follows.
The next three sections are devoted to elementary parts of the final proofs:
\begin{itemize}
\item In Section \ref{extractionsection}, we recall the basic extraction lemmas from \cite{dSPsym}
that will be used throughout the article.
\item In Section \ref{keylemmasection}, we prove the existence of
a linear hyperplane $H$ of $\K^n$ such that the dimension of $S_H$ is small enough, when $S$ is a linear subspace of symmetric or alternating
matrices with bounded rank.
\item Finally, in Section \ref{liftingsection}, we prove the various lifting results that will allow us to recover the structure of $S$
from that of $P(S)$, under specific assumptions on the dimension and on the structure of $P(S)$. In the end of that section,
we shall give a quick inductive proof of Propositions \ref{altrank2}, \ref{symrank2} and \ref{symrank3}.
\end{itemize}

In the final two sections, we combine those results to prove Theorems \ref{maintheoalt} and \ref{maintheosym}
by induction on $n$ and $r$ (in Sections \ref{altproofsection} and \ref{symproofsection}, respectively).

\section{The extraction lemmas}\label{extractionsection}

\subsection{The Schur complement lemma}

The following lemma is classical (see e.g. Lemma 2.1 from \cite{dSPsym}).

\begin{lemme}\label{schurcomplementlemma}
Let $r \in \lcro 1,n-1\rcro$, $A \in \GL_r(\K)$, $B \in \Mat_{n-r,r}(\K)$, $C \in \Mat_{r,n-r}(\K)$ and $D \in \Mat_{n-r}(\K)$.
Then,
$$\rk \begin{bmatrix}
A & C \\
B & D
\end{bmatrix} = r+\rk(D-B A^{-1} C).$$
\end{lemme}

\subsection{The alternating case}

The following result is Lemma 2.4 in \cite{dSPsym}.

\begin{lemme}\label{alternatingcornerlemma}
Let $n$ be an integer such that $n \geq 3$. Let $r$ be an even integer such that $0\leq r<n$.
Let $A \in \Mata_n(\K)$, which we split up as
$$A=\begin{bmatrix}
0 & L & a \\
-L^T & B & C \\
-a & -C^T & 0
\end{bmatrix}$$
with $L \in \Mat_{1,n-2}(\K)$, $C \in \K^{n-2}$, $a \in \K$ and $B \in \Mata_{n-2}(\K)$.
Set $N:=E_{1,n}-E_{n,1}$ and assume that $\rk(A+tN) \leq r$ for all $t \in \K$. \\
Then, $\rk B \leq r-2$.
\end{lemme}

\begin{cor}\label{extractioncoralt}
Let $(i,j)\in \lcro 1,n\rcro^2$ be such that $i \neq j$, and let $M \in \Mata_n(\K)$.
Assume that $\rk(M+t(E_{i,j}-E_{j,i})) \leq r$ for all $t \in \K$.
If we denote by $M_{i,j}$ the submatrix of $M$ obtained by deleting the $i$-th and $j$-th rows and columns,
then $\rk M_{i,j} \leq r-2$.
\end{cor}

\subsection{The symmetric case}

The following result is Lemma 2.5 from \cite{dSPsym}.

\begin{lemme}\label{symmetriccornerlemma2}
Let $n$ be an integer such that $n \geq 3$. Let $r$ be a non-negative integer such that $r<n$.
Let $A \in \Mats_n(\K)$, which we split up as
$$A=\begin{bmatrix}
a & L & b \\
L^T & B & C \\
b & C^T & c
\end{bmatrix}$$
with $L \in \Mat_{1,n-2}(\K)$, $C \in \K^{n-2}$, $(a,b,c) \in \K^3$ and $B \in \Mats_{n-2}(\K)$.
Let $d \in \K$, set
$$N:=\begin{bmatrix}
0 & [0]_{1 \times (n-2)} & 1 \\
[0]_{(n-2) \times 1} & [0]_{(n-2) \times (n-2)} & [0]_{(n-2) \times 1} \\
1 & [0]_{1 \times (n-2)} & d
\end{bmatrix}$$
and assume that $\rk(A+tN) \leq r$ for all $t \in \K$, and that $\# \K>2$.
Then, $\rk B \leq r-2$.
\end{lemme}

\begin{cor}\label{extractioncorsym}
Let $(i,j)\in \lcro 1,n\rcro^2$ be such that $i \neq j$, and let $M \in \Mats_n(\K)$.
Assume that $\rk(M+t(E_{i,j}+E_{j,i})) \leq r$ for all $t \in \K$, and that $\# \K>2$.
If we denote by $M_{i,j}$ the submatrix of $M$ obtained by deleting the $i$-th and $j$-th rows and columns,
then $\rk M_{i,j} \leq r-2$.
\end{cor}

\section{On the matrices of small rank in the translation vector space of a bounded-rank affine space}\label{keylemmasection}

In this part, we shall start from an affine subspace $\calS$ of $\Mats_n(\K)$ or of $\Mata_n(\K)$
with upper-rank $r<n$, and we will try to prove that there exists a linear hyperplane $H$ of $\K^n$
such that $\dim S_H$ is small with respect to $r$, where $S$ denotes the translation vector space of $\calS$.

\subsection{The alternating case}

\begin{prop}\label{keylemmaalternate}
Let $\calS$ be an affine subspace of $\Mata_n(\K)$ with upper-rank $2s$ for some non-zero integer $s$.
Assume that $2s<n-1$, and denote by $S$ the translation vector space of $\calS$.
Then, either $\calS$ is congruent to $\WA_{n,s,1}(\K)$, or
there exists a linear hyperplane $H$ of $\K^n$ such that $\dim S_H<s$.
\end{prop}

Note here that $\calS$ is an affine subspace and not simply a linear subspace.
In this article, we will only need the special case of linear subspaces. Yet, the general case might
be useful in further research on the topic, and it is a more natural framework for that kind of result.

The proof will use a similar strategy as the one of Lemma 2.7 from \cite{dSPboundedrankv2}.

\begin{proof}
Set $r:=2s$.
We assume that $\dim S_H \geq s$ for every linear hyperplane $H$ of $\K^n$,
and in that situation we show that $\calS$ is congruent to $\WA_{n,s,1}(\K)$.
Let us choose a matrix $A \in \calS$ with rank $2s$. Without loss of generality, we can assume that
$$A=\begin{bmatrix}
P & [0]_{r \times (n-r)} \\
[0]_{(n-r) \times r} & [0]_{(n-r) \times (n-r)}
\end{bmatrix} \quad \text{for some $P \in \GL_r(\K) \cap \Mata_r(\K)$.}$$
Now, let $H$ be an arbitrary hyperplane of $\K^n$ that includes $\K^r \times \{0\}$.
The matrices of $S_H$ have the form
$$N=\begin{bmatrix}
[0]_{r \times r} & C(N) \\
-C(N)^T & D(N)
\end{bmatrix} \quad \text{with $C(N) \in \Mat_{r,n-r}(\K)$ and $D(N) \in \Mata_{n-r}(\K)$.}$$
Let $N \in S_H$. Then, we see that $D(N)=0$: indeed Lemma \ref{schurcomplementlemma} applied to $A+N$ yields
$D(N)=-C(N)^T P^{-1} C(N)$, and since $C(N)$ has rank $1$ we deduce that $D(N)=0$ because the rank of the alternating matrix
$D(N)$ must be even.
It follows that
$$\dim S_H=\dim C(S_H).$$
In the next step, we set
$$T_H:=\bigcup_{N \in S_H} \im C(N)$$
and we note that the definition of $S_H$ guarantees that $T_H$ is actually a linear subspace of $\K^r$.
Our aim is to prove that the $T_H$ spaces (when $H$ ranges over the linear hyperplanes of $\K^n$ that include $\K^r \times \{0\}$)
are all equal to a fixed $s$-dimensional totally singular subspace
for the non-degenerate alternating form
$$b : (X,Y)\in (\K^r)^2 \mapsto X^T P^{-1} Y.$$
To do so, consider first two distinct linear hyperplanes $H_1$ and $H_2$ of $\K^n$ that include $\K^r \times \{0\}$.
Then, we prove that $T_{H_1}=T_{H_2}^{\bot_b}$ and $\dim T_{H_1}=s$.
Without loss of generality, we can assume that $H_1$ and $H_2$ are defined, respectively, by the equations
$x_{r+1}=0$ and $x_{r+2}=0$ in the canonical basis.
Then, for $N_1 \in S_{H_1}$ and $N_2 \in S_{H_2}$ we can write
$$C(N_1)=\begin{bmatrix}
C_1(N_1) & [0]_{r \times 1} & [0]_{r \times (n-r-2)}
\end{bmatrix} \quad \text{and} \quad
C(N_2)=\begin{bmatrix}
[0]_{r \times 1} & C_2(N_2) &  [0]_{r \times (n-r-2)}
\end{bmatrix},$$
with $C_1(N_1) \in \K^r$ and $C_2(N_2) \in \K^r$.
Since $D(N_1+N_2)=D(N_1)+D(N_2)=0$, applying Lemma \ref{schurcomplementlemma} to
$A+N_1+N_2$ yields
$$(C(N_1)+C(N_2))^T P^{-1} (C(N_1)+C(N_2))=0,$$
and we deduce, by computing the entry at the $(1,2)$-spot, that $b\bigl(C_1(N_1),C_2(N_2)\bigr)=0$.
Noting that $T_{H_1}=C_1(S_{H_1})$ and $T_{H_2}=C_2(S_{H_2})$,
we deduce that $T_{H_1}$ is $b$-orthogonal to $T_{H_2}$,
whereas $\dim T_{H_1} \geq s$ and $\dim T_{H_2} \geq s$ because of our initial assumption.
By orthogonality theory we find $\dim T_{H_1}+\dim T_{H_2} \leq r=2s$, and hence
$\dim T_{H_1}=\dim T_{H_2}=s$ and $T_{H_1}=T_{H_2}^{\bot_b}$.

Next, we can find a third linear hyperplane $H_3$ that includes $\K^r \times \{0\}$ and is distinct from both
$H_1$ and $H_2$! Applying the above result to both pairs $(H_1,H_3)$ and $(H_2,H_3)$ yields
$T_{H_1}=T_{H_3}^{\bot_b}=T_{H_2}$.

To sum things up, we have exhibited an $s$-dimensional linear subspace $V$ of $\K^r$ such that
$T_H=V$ for every linear hyperplane $H$ of $\K^n$ that includes $\K^r \times \{0\}$.
Performing a harmless congruence, we can now assume that $V=\K^s \times \{0\}$.
In that case, we obtain that $S$ contains \emph{every} alternating matrix of the form
$$\begin{bmatrix}
[0]_{s \times s} & [0]_{s \times s} & [?]_{s \times (n-r)} \\
[0]_{s \times s} & [0]_{s \times s} & [0]_{s \times (n-r)} \\
[?]_{(n-r) \times s} & [0]_{(n-r) \times s} & [0]_{(n-r) \times (n-r)}
\end{bmatrix}.$$
Our ultimate goal is to prove that $\calS$ is included in $\WA_{n,s,1}(\K)$.
In the next step, we demonstrate that $S$ includes $\WA_{n,s,1}(\K)$.
To do so, we use an invariance trick. Fix $k \in \lcro 1,r\rcro$ and
consider the affine space $\calS'$ obtained from $\calS$ by applying the elementary congruence transformation
$(L,C)_k \leftarrow (L,C)_k-(L,C)_n$. Note that this transformation leaves $A$ invariant, whence $A \in \calS'$.
Denoting by $S'$ the translation vector space of $\calS'$, we see that $S'$ still contains
$E_{i,r+1}-E_{r+1,i}$ (because $r \leq n-2$) for all $i \in \lcro 1,s\rcro$.
From the first part of the proof, we know that there exists an $s$-dimensional linear subspace $V'$ of $\K^r$ such that
$T'_H=V'$ for every linear hyperplane $H$ of $\K^n$ that includes $\K^r \times \{0\}$ ($T'_H$ is defined from $S'$ as $T_H$ was defined from $S$).
Taking the hyperplane defined by the equation $x_{r+1}=0$ in the standard basis then yields $V'=\K^s \times \{0\}$.
Let $i \in \lcro 1,s\rcro$. Then, $S'$ contains $E_{i,n}-E_{n,i}$, whence $S$ contains $E_{i,n}+E_{i,k}-E_{n,i}-E_{k,i}$.
Yet, $S$ contains $E_{i,n}-E_{n,i}$, and hence it contains $E_{i,k}-E_{k,i}$.
We conclude that $S$ contains $E_{i,j}-E_{j,i}$ for all $(i,j)\in \lcro 1,n\rcro^2 \setminus \lcro s+1,n\rcro^2$
such that $i \neq j$.

Let $M=(m_{i,j}) \in \calS$. We claim that $m_{i,j}=0$ for all $(i,j)\in \lcro s+1,n\rcro^2$.
Indeed, let us choose an increasing sequence of indices $s+1 \leq j_1<j_2<\cdots <j_s\leq n$,
and denote by $N$ the matrix obtained from $M$ by deleting all rows and columns with index in
$\lcro 1,s\rcro \cup \{ j_1,\dots,j_s\}$.
For all $(t_1,\dots,t_s) \in \K^s$, we know that
the matrix $M+\underset{i=1}{\overset{s}{\sum}} t_i (E_{i,j_i}-E_{j_i,i})$ has rank less than or equal to $2s$.
Applying Lemma \ref{extractioncoralt} inductively yields $\rk N \leq 0$, and hence $N=0$.
As $s<n-s$, varying the sequence $(j_k)_{1 \leq k \leq s}$ shows that
$m_{i,j}=0$ for all $(i,j)\in \lcro s+1,n\rcro^2$.

Therefore, $\calS$ is included in $\WA_{n,s,1}(\K)$. On the other hand, we have shown that $\WA_{n,s,1}(\K) \subset S$, and hence
$\calS=\WA_{n,s,1}(\K)$.
\end{proof}

\subsection{The symmetric case, with $r<n-1$}

Here, we state and prove the equivalent of Proposition \ref{keylemmaalternate} for spaces of symmetric matrices.
We split the discussion into two cases, whether the upper-rank is even or odd.

\begin{Def}
Let $\calS$ be a linear subspace of $\Mats_n(\K)$.
A linear hyperplane $H$ of $\K^n$ is called \textbf{$\calS$-adapted} when it satisfies the following two conditions:
\begin{enumerate}[(a)]
\item The space $S_H$ contains no rank $1$ matrix;
\item If $\K$ has characteristic $2$ then there exists $X \in H$ and $M \in \calS$ such that $X^T MX \neq 0$.
\end{enumerate}
\end{Def}

\begin{prop}\label{keylemmasymmetriceven}
Assume that $\# \K>2$.
Let $\calS$ be an affine subspace of $\Mats_n(\K)$ that is not included in
$\Mata_n(\K)$. Assume that $\urk(\calS \setminus \Mata_n(\K))=2s$ for some non-zero integer $s$ such that $2s<n-1$.
Then, $\calS$ is congruent to a subspace of $\WS_{n,s,0}(\K)$ or there exists
an $\calS$-adapted linear hyperplane $H$ of $\K^n$ such that $\dim S_H<s$.
\end{prop}

\begin{proof}
The chain of arguments is essentially similar to the one of Proposition \ref{keylemmaalternate}.
We assume that there is no $\calS$-adapted linear hyperplane $H$ of $\K^n$ such that $\dim S_H<s$, and we try to prove that
$\calS$ is congruent to a subspace of $\WS_{n,s,0}(\K)$. Set $r:=2s$.

Without loss of generality, we can assume that $\calS$ contains a rank $r$ symmetric matrix
$$A=\begin{bmatrix}
P & [0]_{r \times (n-r)} \\
[0]_{(n-r) \times r} & [0]_{(n-r) \times (n-r)}
\end{bmatrix} \quad \text{for some \emph{non-alternating} matrix $P \in \GL_r(\K) \cap \Mats_r(\K)$.}$$

Let $H$ be an arbitrary hyperplane of $\K^n$ that includes $\K^r \times \{0\}$.
The matrices of $S_H$ have the form
$$N=\begin{bmatrix}
[0]_{r \times r} & C(N) \\
C(N)^T & D(N)
\end{bmatrix} \quad \text{with $C(N) \in \Mat_{r,n-r}(\K)$ and $D(N) \in \Mats_{n-r}(\K)$.}$$
We set $$T_H:=\bigcup_{N \in S_H} \im C(N).$$
As in the proof of Proposition \ref{keylemmaalternate}, we shall prove that
the $D$ map vanishes at every matrix of $S_H$, for every linear hyperplane $H$ of $\K^n$ that includes $\K^r \times \{0\}$,
and that the $T_H$ subspaces, for such $H$, are all equal to some fixed $s$-dimensional totally singular subspace of $\K^r$ for the non-degenerate
symmetric bilinear form
$$b: (X,Y) \mapsto X^T P^{-1}Y.$$

Fix $H$ and $N \in S_H$. Note that $A+N$ is non-alternating, whence, by Lemma \ref{schurcomplementlemma},
$$D(N)=C(N)^T P^{-1} C(N).$$
Applying this to $tN$ for all $t \in \K$ leads, since $\# \K>2$,
to
$$D(N)=0 \quad \text{and} \quad C(N)^T P^{-1} C(N)=0.$$
In particular, $S_H$ contains no rank $1$ matrix and
$\dim T_H=\dim S_H \geq s$. Since $P$ is non-alternating, there exists $X \in H$ such that $X^T A X \neq 0$,
whence $H$ is $\calS$-adapted.
Next, we take distinct linear hyperplanes of $H_1$ and $H_2$ of $\K^n$ that include $\K^r \times \{0\}$.
Then, we prove that $T_{H_1}=T_{H_2}^{\bot_b}$ and $\dim T_{H_1}=s$.

Without loss of generality, we can assume that $H_1$ and $H_2$ are defined, respectively, by the equations
$x_{r+1}=0$ and $x_{r+2}=0$ in the canonical basis.
Then, for $N_1 \in S_{H_1}$ and $N_2 \in S_{H_2}$ we can write
$$C(N_1)=\begin{bmatrix}
C_1(N_1) & [0]_{r \times 1} & [0]_{r \times (n-r-2)}
\end{bmatrix} \quad \text{and} \quad
C(N_2)=\begin{bmatrix}
[0]_{r \times 1} & C_2(N_2) &  [0]_{r \times (n-r-2)}
\end{bmatrix},$$
with $C_1(N_1) \in \K^r$ and $C_2(N_2) \in \K^r$.
In the identity $C(N_1+N_2)^T P^{-1} C(N_1+N_2)=0$, computing the entry at the
$(1,2)$-spot leads to
$$b\bigl(C_1(N_1),C_2(N_2)\bigr)=0.$$
Since $b$ is non-degenerate, we obtain, as in the proof of Proposition \ref{keylemmaalternate}, that
$T_{H_1}=(T_{H_2})^{\bot_b}$ and $\dim T_{H_1}=s$.
Then, using a third hyperplane of the same type, we conclude that $T_{H_1}=T_{H_2}$.

Generalizing this, we conclude that there is an $s$-dimensional linear subspace $V$ of $\K^r$ such that, for every linear
hyperplane $H$ of $\K^n$ that includes $\K^r \times \{0\}$,
we have $T_H=V$ and $D$ vanishes everywhere on $S_H$. Without loss of generality, we can assume that $V=\K^s \times \{0\}$.
From there, by following the chain of arguments of the end of the proof of Proposition \ref{keylemmaalternate},
we successively show:
\begin{itemize}
\item That $S$ contains $E_{i,j}+E_{j,i}$ for all $(i,j)$ in $\lcro 1,n\rcro^2 \setminus \lcro s+1,n\rcro^2$ such that $i \neq j$;
\item That every matrix of $\calS \setminus \Mata_n(\K)$ belongs to $\WS_{n,s,0}(\K)$ (this time, by applying Corollary \ref{extractioncorsym}).
\end{itemize}
Yet, $\Mata_n(\K) \cap \calS$ is included in an affine hyperplane of $\calS$.
Since $\# \K>2$, we deduce that $\calS \setminus \Mata_n(\K)$ generates the affine space $\calS$, and we conclude that
$\calS \subset \WS_{n,s,0}(\K)$.
\end{proof}

Next, we consider the case when the upper-rank is odd.

\begin{prop}\label{keylemmasymmetricodd}
Assume that $\# \K>2$.
Let $\calS$ be an affine subspace of $\Mats_n(\K)$ that is not included in
$\Mata_n(\K)$. Assume that $\urk(\calS \setminus \Mata_n(\K))=2s+1$ for some non-zero integer $s$ such that $2s+1<n-1$.
Then, $\calS$ is congruent to a subspace of $\WS_{n,s,1}(\K)$ or there exists
an $\calS$-adapted linear hyperplane $H$ of $\K^n$ such that $\dim S_H<s$.
\end{prop}

\begin{proof}
We assume that there is no $\calS$-adapted linear hyperplane $H$ of $\K^n$ such that $\dim S_H<s$, and we try to prove that
$\calS$ is congruent to a subspace of $\WS_{n,s,1}(\K)$. Set $r:=2s+1$.

Without loss of generality, we can assume that $\calS$ contains a rank $r$ symmetric matrix
$$A=\begin{bmatrix}
P & [0]_{r \times (n-r)} \\
[0]_{(n-r) \times r} & [0]_{(n-r) \times (n-r)}
\end{bmatrix} \quad \text{for some \emph{non-alternating} matrix $P \in \GL_r(\K) \cap \Mats_r(\K)$}.$$

Let $H$ be an arbitrary linear hyperplane of $\K^n$ that includes $\K^r \times \{0\}$.
The matrices of $S_H$ have the form
$$N=\begin{bmatrix}
[0]_{r \times r} & C(N) \\
C(N)^T & D(N)
\end{bmatrix} \quad \text{with $C(N) \in \Mat_{r,n-r}(\K)$ and $D(N) \in \Mats_{n-r}(\K)$.}$$
We set $$T_H:=\bigcup_{N \in S_H} \im C(N).$$
As in the proof of Proposition \ref{keylemmasymmetriceven}, we obtain that $D$ vanishes everywhere on $S_H$ and that
$$\forall N \in S_H, \; C(N)^T P^{-1} C(N)=0.$$
As in the proof of Proposition \ref{keylemmasymmetriceven}, we deduce that $H$ is $\calS$-adapted and $\dim T_H=\dim S_H \geq s$.
We shall now prove that the $T_H$ subspaces are all equal to some $s$-dimensional totally singular subspace of $\K^s$ for the non-degenerate
symmetric bilinear form
$$b: (X,Y) \mapsto X^T P^{-1}Y.$$
Firstly, we take distinct linear hyperplanes of $H_1$ and $H_2$ of $\K^n$ that include $\K^r \times \{0\}$.
We shall demonstrate that $T_{H_1}=T_{H_2}$.
Our first step consists in proving that $T_{H_1}$ and $T_{H_2}$ are $b$-orthogonal and that they lie
inside the isotropy cone of the quadratic form
$$q : X \mapsto b(X,X).$$

Without loss of generality, we can assume that $H_1$ and $H_2$ are defined, respectively, by the equations
$x_{r+1}=0$ and $x_{r+2}=0$ in the canonical basis.
Then, for $N_1 \in S_{H_1}$ and $N_2 \in S_{H_2}$ we can write
$$C(N_1)=\begin{bmatrix}
C_1(N_1) & [0]_{r \times 1} & [0]_{r \times (n-r-2)}
\end{bmatrix} \quad \text{and} \quad
C(N_2)=\begin{bmatrix}
[0]_{r \times 1} & C_2(N_2) &  [0]_{r \times (n-r-2)}
\end{bmatrix},$$
with $C_1(N_1) \in \K^r$ and $C_2(N_2) \in \K^r$. \\
From the identity $D(N_1+N_2)=C(N_1+N_2)^T P^{-1} C(N_1+N_2)$, we deduce that
$$b\bigl(C_1(N_1),C_2(N_2)\bigr)=0, \; b\bigl(C_1(N_1),C_1(N_1)\bigr)=0 \quad \text{and} \quad
b\bigl(C_2(N_2),C_2(N_2)\bigr)=0.$$
Hence, $T_{H_1}$ and $T_{H_2}$ are $b$-orthogonal and lay in the isotropy cone of $q$.

Assume now that $\dim (T_{H_1}+T_{H_2}) \geq s+1$. We can choose
a third hyperplane $H_3$ that is distinct from $H_1$ and $H_2$ and that includes $\K^r \times \{0\}$
(e.g., we define $H_3$ by the equation $x_{r+1}=x_{r+2}$ in the standard basis).
Then, we find that $T_{H_3} \subset T_{H_1}^{\bot_b} \cap T_{H_2}^{\bot_b} =(T_{H_1}+T_{H_2})^{\bot_b}$,
and as $\dim T_{H_3}=s$ it follows that $\dim (T_{H_1}+T_{H_2})=s+1$.
From there, we shall prove that $T_{H_1}+T_{H_2}$ is a totally singular subspace for $b$.

Since $\dim T_{H_1} \geq s$ and $\dim T_{H_2} \geq s$, this leaves us with three cases
to study:

\begin{itemize}
\item $T_{H_1}$ and $T_{H_2}$ are distinct and have dimension $s$. Then, $T_{H_1} \cap T_{H_2}$ has dimension $s-1$, and
we can find a basis $(X_1,\dots,X_{s+1})$ of $T_{H_1}+T_{H_2}$ in which $X_1,\dots,X_{s-1}$ all belong to $T_{H_1} \cap T_{H_2}$,
$X_s$ belongs to $T_{H_1}$ and $X_{s+1}$ belongs to $T_{H_2}$. Then, it is clear that the $X_i$ vectors are $b$-isotropic and pairwise
$b$-orthogonal, which yields that $T_{H_1}+T_{H_2}$ is totally singular for $b$.

\item $T_{H_1}$ and $T_{H_2}$ are equal and have dimension $s+1$: it is then straightforward that $T_{H_1}+T_{H_2}=T_{H_1}=T_{H_2}$
is totally singular for $b$;

\item One of $T_{H_1}$ and $T_{H_2}$, say $T_{H_1}$, has dimension $s+1$, and the other one has dimension $s$.
Then, $T_{H_1}+T_{H_2}=T_{H_1}$ and hence
we find a basis $(X_1,\dots,X_{s+1})$ of $T_{H_1}$
in which the first $s$ vectors belong to $T_{H_2}$.
Again, it is clear that the $X_i$ vectors are $b$-isotropic and pairwise $b$-orthogonal, whence
$T_{H_1}+T_{H_2}$ is totally singular for $b$.
\end{itemize}

In any case, we have an $(s+1)$-dimensional linear subspace of $\K^{2s+1}$ that is totally $b$-singular, which
contradicts the fact that $b$ is non-degenerate.
Hence $\dim(T_{H_1}+T_{H_2}) \leq s$, and since $\dim T_{H_1} \geq s$ and $\dim T_{H_2}\geq s$, it follows that
$T_{H_1}=T_{H_1}+T_{H_2}=T_{H_2}$, as claimed.

From there, we can follow exactly the same line of reasoning as in the proof of Proposition \ref{keylemmaalternate}:
we have found an $s$-dimensional linear subspace $V$ of $\K^{2s+1}$ such that $T_H=V$ for every linear hyperplane $H$ of $\K^n$
that includes $\K^r \times \{0\}$. We lose no generality in assuming that $V=\K^s \times \{0\}$.
Using the invariance trick, we deduce that $S$ contains every matrix of $\WS_{n,s,0}(\K)$ with diagonal zero.
Writing every matrix $M$ of $\calS$ as
$$M=\begin{bmatrix}
[?]_{s \times s} & [?]_{s \times (n-s)} \\
[?]_{(n-s) \times s} & J(M)
\end{bmatrix}$$
with $J(M) \in \Mats_{n-s}(\K)$, we use Corollary \ref{extractioncorsym} together with $s<n-s-1$ to find that
every $2$ by $2$ submatrix of $J(M)$ is singular \emph{provided that $M$ be non-alternating.}
Hence for all $M \in \calS \setminus \Mata_n(\K)$, the matrix $J(M)$ has rank at most $1$.

Assume now that there are (distinct) matrices $M_1$ and $M_2$ in $\calS \setminus \Mata_n(\K)$ such that
$J(M_1)$ and $J(M_2)$ are non-collinear. Then, for some $Q \in \GL_{n-s}(\K)$ and some pair $(a,b) \in (\K \setminus \{0\})^2$,
we have $J(M_1)=a\,Q E_{1,1} Q^T$ and $J(M_2)=b\,Q E_{2,2} Q^T$.
Choosing $t \in \K \setminus \{0,1\}$, we see that $J(tM_1+(1-t)M_2)=Q\,\bigl(ta E_{1,1}+(1-t)b E_{2,2}\bigr)\,Q^T$, and hence
$J(tM_1+(1-t)M_2)$ has rank $2$.

Now, we obtain a contradiction by distinguishing between two cases:
\begin{itemize}
\item If $\# \K>3$, then the line going through $M_1$ and $M_2$ contains at most one point of $\Mata_n(\K)$,
whence we can choose $t \in \K \setminus \{0,1\}$ such that $tM_1+(1-t)M_2 \not\in \Mata_n(\K)$. This is a contradiction.

\item Assume now that $\# \K=3$. Then $\calS \cap \Mata_n(\K) \subset \{0\}$, and hence $\urk J(\calS) \leq 1$.
Hence, we obtain a contradiction by taking $t:=-1$ and by applying the above result to $tM_1+(1-t)M_2$.
\end{itemize}
It follows that there is a rank $1$ symmetric matrix $U \in \Mats_{n-s}(\K)$ such that
$J(M) \in \Vect(U)$ for all $M \in S \setminus \Mata_n(\K)$. Since the affine space
$\calS$ is generated by its non-alternating matrices we deduce that $J(\calS) \subset \Vect(U)$.
As $\Vect(U)$ is congruent to $\Vect(E_{1,1})$, we conclude that $\calS$ is congruent to a subspace of $\WS_{n,s,1}(\K)$.
\end{proof}

\subsection{The symmetric case, with $r=n-1$}

In the preceding two results, we always assumed that the upper-rank was less than $n-1$, a condition that was crucial in the proofs.
Yet, we shall need to deal with the case when the upper-rank is $n-1$. In that situation, our result
is not as good as the previous ones, but it will be good enough in order to prove Theorem \ref{maintheosym}.
First of all, we recall the known result in the case when $\K$ has characteristic not $2$:
this is deduced from Lemma 2.9 of \cite{dSPsym}.

\begin{prop}\label{oldkeylemmasymmetricsubfull}
Assume that $\K$ has characteristic not $2$.
Let $\calS$ be an affine subspace of $\Mats_n(\K)$, with translation vector space denoted by $S$.
Assume that $\urk \calS \leq n-1$.
Then, there exists a linear hyperplane $H$ of $\K^n$ such that $\dim S_H \leq \bigl\lfloor \frac{n-1}{2}\bigr\rfloor$.
\end{prop}

Next, we give the result that will be used in the general case.

\begin{prop}\label{keylemmasymmetricsubfull}
Let $\calS$ be an affine subspace of $\Mats_n(\K)$ that is not included in $\Mata_n(\K)$.
Assume that $\# \K>2$. Assume that all the matrices of $\calS$ are singular and that $n \geq 5$.
Denote by $\epsilon$ the remainder of $n-1$ mod $2$, and by $s$ the quotient.
Then, $\calS$ is congruent to a subspace of $\WS_{n,s,\epsilon}$ or there exists an $\calS$-adapted linear hyperplane
$H$ of $\K^n$ such that $\dim S_H \leq n-3$.
\end{prop}

\begin{proof}
Since $n \geq 5$, we see that $\lfloor \frac{n-1}{2}\rfloor \leq n-3$, and
hence the result follows directly from Proposition \ref{oldkeylemmasymmetricsubfull} if $\K$ has characteristic not $2$.

In the remainder of the proof, we assume that $\K$ has characteristic $2$
and that there is no $\calS$-adapted linear hyperplane $H$ of $\K^n$ such that $\dim S_H \leq n-3$.
Denote by $r$ the upper-rank of $\calS \setminus \Mata_n(\K)$. If
$r<n-1$ then $\lfloor \frac{r}{2}\rfloor -1 \leq n-3$, and hence the result follows from one of
Propositions \ref{keylemmasymmetriceven} and \ref{keylemmasymmetricodd}.
In the rest of the proof, we assume that $r=n-1$.
Hence, we have a rank $n-1$ non-alternating matrix $M_1$ in $\calS$.
Without loss of generality, we can assume that
$$M_1=\begin{bmatrix}
P & [0]_{(n-1) \times 1} \\
[0]_{1 \times (n-1)} & 0
\end{bmatrix}$$
where $P$ is an invertible matrix of $\Mats_{n-1}(\K) \setminus \Mata_{n-1}(\K)$.
Note that
$$G:=\bigl\{X \in \K^{n-1} : \; X^T P^{-1} X=0\bigr\} \times \{0\},$$
is a linear subspace of $\K^n$ since $\K$ has characteristic $2$.

Set $H:=\K^{n-1} \times \{0\}$.  We claim that $H$ is $\calS$-adapted and $S_H\,e_n \subset G$.
Let $N \in S_H$. We write
$$N=\begin{bmatrix}
[0]_{(n-1) \times (n-1)} & C(N) \\
C(N)^T & D(N)
\end{bmatrix}$$
with $C(N) \in \K^{n-1}$ and $D(N) \in \K$.
As in the previous two proofs, we find that
$$D(N)=0 \quad \text{and} \quad C(N)^T P^{-1} C(N)=0.$$
Hence, $S_H e_n \subset G$ and $\dim S_H=\dim (S_H e_n)$.
Moreover, we see that $S_H$ contains no rank $1$ matrix and that the quadratic form
that is attached to $M_1$ does not vanish everywhere on $H$, whence $H$ is $\calS$-adapted.

Since $P$ is not alternating, the matrix $P^{-1}$ is not alternating either, whence $G$ is a proper linear subspace of $\K^{n-1} \times \{0\}$.
Yet $\dim S_H\, e_n=\dim S_H \geq n-2$, and we obtain the claimed equality $S_H e_n=G$.

Without loss of generality, we can now assume that $G \subset \K^{n-2} \times \{0\}$.
Then, we learn that $S$ contains $E_{i,n}+E_{n,i}$ for all $i \in \lcro 1,n-2\rcro$.
Next, fix $j \in \lcro 1,n-1\rcro$ and consider the affine space $\calS'$ that is deduced from $\calS$
by performing the elementary congruence transformation $(L,C)_j \leftarrow (L,C)_j- (L,C)_n$.
Denote by $S'$ its translation vector space, and note that $\urk \calS'=\urk \calS$ and $\calS'$ still contains $M_1$.
Hence, the above proof shows that $S'$ contains $E_{i,n}+E_{n,i}$ for all $i \in \lcro 1,n-2\rcro$.
Fix $i \in \lcro 1,n-2\rcro$. Then, $S$ contains $E_{i,n}+E_{i,j}+E_{n,i}+E_{j,i}$.
Since $S$ contains $E_{i,n}+E_{n,i}$, we conclude that $S$ contains $E_{i,j}+E_{j,i}$.

Hence, we have shown that $S$ contains $E_{i,j}+E_{j,i}$ for all $(i,j)\in \lcro 1,n\rcro^2 \setminus \{n-1,n\}^2$ such that $i \neq j$.

Since $P$ is non-alternating, we deduce that $\calS$ contains a non-zero diagonal matrix of the form
$D=\Diag(a_1,\dots,a_{n-1},a_n)$ with $a_n=0$.
Since $n \geq 4$, we can choose a permutation $\sigma$ of $\lcro 1,n-1\rcro$ such that
$a_{\sigma(n-1)} \neq 0$ and $\sigma(1) \neq n-1$.
Then, for each integer $i$ between $1$ and $\left\lfloor \frac{n}{2}\right\rfloor-1$, we
choose $t_i \in \K$ such that $t_i^2 \neq a_{\sigma(2i)} a_{\sigma(2i+1)}$.
One checks that the matrix
$$D+(E_{\sigma(1),n}+E_{n,\sigma(1)})
+\underset{i=1}{\overset{\left\lfloor \frac{n}{2}\right\rfloor-1}{\sum}} t_i (E_{\sigma(2i),\sigma(2i+1)}+E_{\sigma(2i+1),\sigma(2i)})$$
is invertible, which contradicts our assumption because it belongs to $\calS$.
\end{proof}

\section{Lifting results}\label{liftingsection}

\subsection{Lifting results of the first kind}

\begin{prop}\label{lifting1alt}
Let $n$ and $r$ be positive integers with $r$ odd and $n>r$. Let
$V$ be a linear subspace of $\WA_{n-1,0,r}(\K)$, and $f : V \rightarrow \K^{n-1}$ be a linear map.
Assume that every matrix of
$$S:=\biggl\{\begin{bmatrix}
A & f(A) \\
-f(A)^T & 0
\end{bmatrix} \mid A \in V\biggr\}$$
has rank less than $r$ and that $\dim V>\dbinom{r-1}{2}+2$.
Then, $S$ is congruent to a subspace of $\WA_{n,0,r}(\K)$.
\end{prop}

\begin{proof}
For all $A \in V$, let us write $A=\begin{bmatrix}
K(A) & [0]_{r \times (n-1-r)} \\
[0]_{(n-1-r) \times r} & [0]_{(n-1-r) \times (n-1-r)}
\end{bmatrix}$ with $K(A) \in \Mata_r(\K)$. Set $V':=K(V)$, so that
$$\dim V'=\dim V>\dbinom{r-1}{2}+2.$$
We have a linear map $C : V' \rightarrow \K^{n-1}$ such that $\forall A \in V, \; f(A)=C(K(A))$.
For $N \in V'$, let us split
$$C(N)=\begin{bmatrix}
C_1(N) \\
C_2(N)
\end{bmatrix} \quad \text{with $C_1(N) \in \K^r$ and $C_2(N) \in \K^{n-1-r}$.}$$

\vskip 2mm
\noindent \textbf{Step 1: $C_2=0$.} \\
For every $N \in V'$, we note that
$$\rk \begin{bmatrix}
N & [0]_{r \times (n-r-1)} & C_1(N) \\
[0]_{(n-r-1) \times r} & [0]_{(n-r-1) \times (n-r-1)} & C_2(N) \\
-C_1(N)^T & -C_2(N)^T & 0
\end{bmatrix} \geq \rk N+\rk C_2(N)+\rk (-C_2(N))^T$$
which yields $C_2(N)=0$ whenever $\rk N=r-1$.
Yet, by Corollary 3.1 from \cite{dSPsym}, the space $V'$ is spanned by its rank $r-1$ matrices.
It follows that $C_2=0$.

\vskip 2mm
\noindent \textbf{Step 2: $C_1$ is range-compatible.} \\
Given a linear hyperplane $G$ of $\K^r$, we prove that $C_1(N) \in G$ for all $N \in V'$ such that $\im N \subset G$.
Using a congruence transformation, we see that it suffices to consider the case when $G=\K^{r-1} \times \{0\}$.
Denote by $W$ the linear subspace of $V'$ consisting of its matrices $N$ such that $\im N \subset G$.
For all such $N$, let us write
$$N=\begin{bmatrix}
J(N) & [0]_{(r-1) \times 1} \\
[0]_{1 \times (r-1)} & 0
\end{bmatrix} \quad \text{with $J(N) \in \Mata_{r-1}(\K)$.}$$
Then, we have a linear form $\gamma : J(V') \rightarrow \K$ such that
$\gamma(M)$ is the last entry of $C_1\biggl(\begin{bmatrix}
M & [0]_{(r-1) \times 1} \\
[0]_{1 \times (r-1)} & 0
\end{bmatrix}\biggr)$ for all $M \in J(V')$.
Hence, for all $M \in J(V')$, the space $S$ contains a matrix of the form
$$\begin{bmatrix}
M & [0]_{(r-1) \times 1} & [0]_{(r-1) \times (n-r-1)} & [?]_{(r-1) \times 1} \\
[0]_{1 \times (r-1)} & 0 & [0]_{1 \times (n-r-1)} & \gamma(M) \\
[0]_{(n-r-1) \times (r-1)} & [0]_{(n-r-1) \times 1} & [0]_{(n-r-1) \times (n-r-1)} & [0]_{(n-r-1) \times 1} \\
[?]_{1 \times (r-1)} & -\gamma(M) & [0]_{1 \times (n-r-1)} & 0
\end{bmatrix}$$
and such a matrix has rank $r+1$ if $\rk M = r-1$ and $\gamma(M)\neq 0$.
Hence, $\gamma$ vanishes at every rank $r-1$ matrix of $J(V')$.
Yet,
$$\dim J(V') \geq \dim V'-(r-1) \geq \dbinom{r-2}{2}+2.$$
Hence, by Corollary 3.1 in \cite{dSPsym}, the vector space $J(V')$ is spanned by its matrices with rank $r-1$, which leads to $\gamma=0$.
This proves the claimed result, that is $C_1(N) \in G$ for all $N \in V'$ such that $\im N \subset G$.

Finally, for any $N \in V'$, we can find linear hyperplanes $G_1,\dots,G_k$ of $\K^r$ such that
$\im N=\underset{i=1}{\overset{k}{\bigcap}} G_i$, and the previous step yields
$C_1(N) \in \underset{i=1}{\overset{k}{\bigcap}} G_i=\im N$. Hence, $C_1$ is range-compatible.

\vskip 2mm
\noindent \textbf{Step 3: The final reduction.} \\
We know that $C_1$ is a range-compatible linear map from $V'$ to $\K^r$.
Yet, $V'$ is a linear subspace of $\Mata_r(\K)$ with codimension at most $r-4$.
By Theorem 1.6 of \cite{dSPRCsym}, the map $C_1$ is local, i.e.\ there exists $Y \in \K^r$ such that
$$\forall N \in V', \; C_1(N)=NY.$$
Setting
$$Q:=\begin{bmatrix}
I_r & [0]_{r \times (n-r-1)} & -Y \\
[0]_{(n-r-1) \times r} & I_{n-r-1} & [0]_{(n-r-1) \times 1} \\
[0]_{1 \times r} & [0]_{1 \times (n-r-1)} & 1
\end{bmatrix} \in \GL_n(\K),$$
we conclude that $Q^T S Q \subset \WA_{n,0,r}(\K)$, whence $S$ is congruent to a subspace of $\WA_{n,0,r}(\K)$.
\end{proof}

\begin{prop}\label{lifting1sym}
Let $n$ and $r$ be positive integers with $n>r$. Let
$V$ be a linear subspace of $\WS_{n-1,0,r}(\K)$, and let $f : V \rightarrow \K^{n-1}$ and $b : V \rightarrow \K$ be linear maps.
Assume that every matrix of
$$S:=\biggl\{\begin{bmatrix}
M & f(M) \\
f(M)^T & b(M)
\end{bmatrix} \mid M \in V\biggr\}$$
has rank less than or equal to $r$ and that $\dim V>\dbinom{r}{2}+2$.
Assume also that $\# \K>2$.
Then, $S$ is congruent to a subspace of $\WS_{n,0,r}(\K)$.
\end{prop}

\begin{proof}
The proof is very similar to the one of Proposition \ref{lifting1alt}.
Again, for all $M \in V$, we write
$$M=\begin{bmatrix}
K(M) & [0]_{r \times (n-r)} \\
[0]_{(n-r) \times r} & [0]_{(n-r) \times (n-r)}
\end{bmatrix} \quad \text{with $K(M) \in \Mats_r(\K)$,}$$
and we set $V':=K(V)$, so that $\dim V'=\dim V>\dbinom{r}{2}+2$.
We have a linear map $C : V' \rightarrow \K^{n-1}$ such that $\forall M \in V, \; f(M)=C(K(M))$.
For $N \in V'$, let us split
$$C(N)=\begin{bmatrix}
C_1(N) \\
C_2(N)
\end{bmatrix} \quad \text{with $C_1(N) \in \K^r$ and $C_2(N) \in \K^{n-1-r}$.}$$

\vskip 2mm
\noindent \textbf{Step 1: $C_2=0$.} \\
The proof is similar to the corresponding step in the proof of Proposition \ref{lifting1alt},
using Corollary 5.1 of \cite{dSPsym} this time around.

\vskip 2mm
\noindent \textbf{Step 2: $C_1$ is range-compatible.} \\
Again, the proof of this claim is similar to the corresponding one for  Proposition \ref{lifting1alt},
using Corollary 5.1 of \cite{dSPsym}.

\vskip 2mm
\noindent \textbf{Step 3: Reduction to the case when $C_1=0$.} \\
We note that $V'$ is a linear subspace of $\Mats_r(\K)$ with codimension at most $r-3$.
Since $C_1$ is a range-compatible linear map, Corollary 1.6 of \cite{dSPRCsym}
shows that $C_1$ is local, yielding a vector $Y \in \K^r$ such that
$$\forall N \in V', \; C_1(N)=NY.$$
Then, by setting
$$Q:=\begin{bmatrix}
I_r & [0]_{r \times (n-r-1)} & -Y \\
[0]_{(n-r-1) \times r} & I_{n-r-1} & [0]_{(n-r-1) \times 1} \\
[0]_{1 \times r} & [0]_{1 \times (n-r-1)} & 1
\end{bmatrix} \in \GL_n(\K)$$
and, by replacing $S$ with $S':=Q^T S Q$, we see that the basic assumptions are still satisfied, but now we have
$C_1=0$. Thus, in the remainder of the proof we assume that $C_1=0$.

\vskip 2mm
\noindent \textbf{Step 4: $b=0$.} \\
Now, $f=0$. It is then obvious that the linear map $b$ vanishes at every rank $r$ matrix of $V'$.
Yet, using Corollary 5.1 from \cite{dSPsym} once more, we know that $V'$ is spanned by
its rank $r$ matrices, and we conclude that $b=0$. Therefore,
in our reduced situation we have shown that $S \subset \WS_{n,0,r}(\K)$, which completes the proof.
\end{proof}

\subsection{Lifting results of the second kind}

\begin{prop}\label{lifting2alt}
Let $n$ and $s$ be positive integers such that $2s<n-1$.
Let $\calS$ be a linear subspace of $\Mata_n(\K)$.
Let us write every matrix $M$ of $\calS$ as
$$M=\begin{bmatrix}
P(M) & [?]_{(n-1) \times 1} \\
[?]_{1 \times (n-1)} & 0
\end{bmatrix} \quad \text{with $P(M) \in \Mata_{n-1}(\K)$.}$$
Assume that $P(\calS) \subset \WA_{n-1,s,1}(\K)$, that
$\dim P(\calS)>a_{n-1,s,1}-(n-s-3)$, and that $\urk \calS \leq 2s$.
Then, $\calS$ is congruent to a subspace of $\WA_{n,s,1}(\K)$.
\end{prop}

\begin{proof}
For all $M \in \calS$, we split
$$M=\begin{bmatrix}
[?]_{s \times s} & -B(M)^T & [?]_{s \times 1} \\
B(M) & [0]_{(n-1-s) \times (n-1-s)} & C(M) \\
[?]_{1 \times s} & -C(M)^T & 0
\end{bmatrix} \quad \text{with $B(M) \in \Mat_{n-1-s,s}(\K)$ and $C(M) \in \K^{n-1-s}$.}$$
Let us set
$$\calT:=\Bigl\{\begin{bmatrix}
B(M) & C(M)
\end{bmatrix} \mid M \in \calS\Bigr\}.$$
The assumption $\dim P(\calS) >a_{n-1,s,1}-(n-s-3)$ leads to
$$\dim B(\calS) \geq s(n-s-1)-(n-s-4),$$
and hence
$$\dim \calT \geq s(n-s-1)-(n-s-1)+3.$$

Next, for all $M \in \calS$, we see that
$$\rk M \geq \rk \begin{bmatrix}
B(M) & C(M)
\end{bmatrix}+\rk \begin{bmatrix}
B(M) & C(M)
\end{bmatrix}^T=2 \rk \begin{bmatrix}
B(M) & C(M)
\end{bmatrix}.$$
Hence,
$$\urk \calT \leq s.$$
Moreover $n-s-1 \geq s+1$ since $2s\leq n-2$.

Thus, we can apply the Atkinson-Lloyd theorem to $\calT$:
in virtue of Theorem 1.5 of \cite{dSPboundedrankv2}, there are two possibilities.
\begin{itemize}
\item Either $n-s-1=s+1$ and there is a non-zero vector $Y \in \K^{n-s-1}$ such that $Y^T N=0$ for all $N \in \calT$:
it would follow that $Y^T B(M)=0$ for all $M \in \calS$, leading to $\dim B(\calS) \leq s(n-s-1)-s=s(n-s-1)-(n-s-1)+1$,
which contradicts a previous result.

\item Or there exists a non-zero vector $X \in \K^{s+1}$ such that $NX=0$ for all $N \in \calT$. \\
If the last entry of $X$ equals zero, we write $X=\begin{bmatrix}
Z \\
0
\end{bmatrix}$ with $Z \in \K^s \setminus \{0\}$ and we learn that $B(M)Z=0$ for all $M \in \calS$.
Again, this would lead to $\dim B(\calS) \leq s(n-s-1)-(n-s-1)$, in contradiction with our assumptions.
Hence, the last entry of $X$ is non-zero, and without loss of generality we can assume that it equals $-1$.
Then, we write $X=\begin{bmatrix}
Y \\
-1
\end{bmatrix}$ with $Y \in \K^s$.
\end{itemize}

We have found a vector $Y \in \K^s$ such that $C(M)=B(M)Y$ for all $M \in \calS$.
Setting
$$Q:=\begin{bmatrix}
I_s & [0]_{s \times (n-1-s)} & -Y \\
[0]_{(n-1-s) \times s} & I_{n-1-s} & [0]_{(n-1-s) \times 1} \\
[0]_{1 \times s} & [0]_{1 \times (n-1-s)} & 1
\end{bmatrix},$$
we deduce that
$$Q^T \calS Q \subset \WA_{n,s,1}(\K),$$
which is the claimed result.
\end{proof}

For symmetric matrices, we give two results, one for even ranks and one for odd ranks:

\begin{prop}\label{lifting2symeven}
Let $n$ and $s$ be positive integers such that $2s<n-1$.
Let $\calS$ be an affine subspace of $\Mats_n(\K)$. Assume that $\K$ has more than $2$ elements.
Let us write every matrix $M$ of $\calS$ as
$$M=\begin{bmatrix}
P(M) & [?]_{(n-1) \times 1} \\
[?]_{1 \times (n-1)} & 0
\end{bmatrix} \quad \text{with $P(M) \in \Mats_{n-1}(\K)$.}$$
Assume that $P(\calS) \subset \WS_{n-1,s,0}(\K)$, that
$\dim P(\calS)>s_{n-1,s,0}-(n-s-3)$, and that $\urk \calS \leq 2s$.
Then, $\calS$ is congruent to a subspace of $\WS_{n,s,0}(\K)$.
\end{prop}

\begin{proof}
For all $M \in \calS$, we split
$$M=\begin{bmatrix}
[?]_{s \times s} & B(M)^T & [?]_{s \times 1} \\
B(M) & [0]_{(n-1-s) \times (n-1-s)} & C(M) \\
[?]_{1 \times s} & C(M)^T & a(M)
\end{bmatrix},$$
with $B(M) \in \Mat_{n-1-s,s}(\K)$, $C(M) \in \K^{n-1-s}$ and $a(M) \in \K$.
Then, with exactly the same proof as in the one of Proposition \ref{lifting1alt},
we find that
$$\dim B(\calS) \geq s(n-1-s)-(n-s-4)$$
and we reduce the situation to the one where $C(M)=0$ for all $M \in \calS$.
In that reduced situation, it remains to prove that $a=0$.
For all $M \in \calS$, we now have
$$\rk M \geq \rk B(M)+\rk B(M)^T+\rk a(M)$$
and hence if $\rk B(M)=s$ then $a(M)=0$.

Assume that $a$ is non-zero. Then, we have an affine subspace $\calS'$ of $\calS$, with codimension at most $1$,
on which $a$ is constant and non-zero. Then, $\urk B(\calS') \leq s-1$ and
$$\dim B(\calS') \geq \dim B(\calS)-1 \geq (s-1)(n-1-s)+1.$$
Since $n-1-s \geq s$, we find a contradiction by applying Flanders's theorem for affine subspaces
(Theorem 6 of \cite{dSPaffpres}).

Hence, $a=0$, and we conclude that $\calS \subset \WS_{n,s,0}(\K)$.
\end{proof}

\begin{prop}\label{lifting2symodd}
Let $n$ and $s$ be positive integers such that $2s+1<n-1$.
Let $\calS$ be an affine subspace of $\Mats_n(\K)$. Assume that $\# \K>2$.
Let us write every matrix $M$ of $\calS$ as
$$M=\begin{bmatrix}
P(M) & [?]_{(n-1) \times 1} \\
[?]_{1 \times (n-1)} & 0
\end{bmatrix} \quad \text{with $P(M) \in \Mats_{n-1}(\K)$.}$$
Assume that $P(\calS) \subset \WS_{n-1,s,1}(\K)$, that
$\dim P(\calS)>s_{n-1,s,1}-(n-s-5)$, and that $\urk \calS \leq 2s+1$.
Then, $\calS$ is congruent to a subspace of $\WS_{n,s,1}(\K)$.
\end{prop}

\begin{proof}
For all $M \in \calS$, we split
$$M=\begin{bmatrix}
[?]_{s \times s} & [?]_{s \times 1} & B(M)^T & [?]_{s \times 1} \\
[?]_{1 \times s} & a(M) & [0]_{1 \times (n-s-2)} & b(M) \\
B(M) & [0]_{(n-s-2) \times 1} & [0]_{(n-s-2) \times (n-s-2)} & C(M) \\
[?]_{1 \times s} & b(M) & C(M)^T & c(M)
\end{bmatrix}$$
with $B(M) \in \Mat_{n-s-2,s}(\K)$, $C(M) \in \K^{n-s-2}$, and scalars $a(M)$, $b(M)$ and $c(M)$.

Then, $n-s-2 \geq s+1$ and $\dim B(\calS) \geq s(n-s-2)-(n-s-2)+4$.
Moreover, for all $M \in \calS$, we see that
$$2s+2>\rk M \geq \rk \begin{bmatrix}
B(M) & C(M)
\end{bmatrix}+\rk \begin{bmatrix}
B(M) & C(M)
\end{bmatrix}^T$$
and hence $\rk \begin{bmatrix}
B(M) & C(M)
\end{bmatrix} \leq s$.
Then, with the same line of reasoning as in the proof of Proposition \ref{lifting2alt},
we obtain a vector $Y \in \K^s$ such that
$$\forall M \in \calS, \; C(M)=B(M)Y,$$
and then we use a congruence transformation to reduce the situation to the one where $C(M)=0$ for all $M \in \calS$.

Next, for $M \in \calS$, we set
$$J(M):=\begin{bmatrix}
a(M) & b(M) \\
b(M) & c(M)
\end{bmatrix} \in \Mats_2(\K).$$
Again, we find that for all $M \in \calS$,
$$\rk M \geq 2\rk B(M)+\rk J(M)$$
and it follows that $\rk J(M) \leq 1$ whenever $\rk B(M)=s$.

We claim that $\urk J(\calS) \leq 1$. Assume that the contrary holds and choose $M_1 \in \calS$ such that
$\rk J(M_1)=2$. Then, $\calS':=\bigl\{M \in \calS : \; J(M)=J(M_1)\bigr\}$ is an affine subspace of $\calS$ and
$$\dim \calS' \geq \dim \calS-3.$$
It follows that $B(\calS')$ is an affine subspace of $B(\calS)$ with upper-rank less than $s$
and
$$\dim B(\calS') \geq \dim B(\calS)-3 \geq s(n-s-2)-(n-s-2)+1.$$
This would contradict Flanders's theorem for affine subspaces.

We conclude that $\urk J(\calS) \leq 1$. By Theorem 1.4 of \cite{dSPsym}, it follows that $\dim J(\calS) \leq 1$, and that
either $J(\calS)$ consists of a single rank $1$ matrix
or it is congruent to $\Vect(E_{1,1})$. In any case, $J(\calS)$ is congruent to a subspace of $\Vect(E_{1,1})$,
and we conclude that $\calS$ is congruent to a subspace of $\WS_{n,s,1}(\K)$.
\end{proof}

The next lifting lemmas are to be used in the proofs of Propositions \ref{altrank2}, \ref{symrank2} and \ref{symrank3}.
This time around, we shall only deal with linear subspaces, but we will not exclude fields with two elements.

\begin{lemme}\label{lifting2altrank2}
Let $n$ be an integer such that $n\geq 3$.
Let $S$ be a linear subspace of $\Mata_n(\K)$ such that $\urk S \leq 2$.
Let us write every matrix $M$ of $S$ as
$$M=\begin{bmatrix}
P(M) & [?]_{(n-1) \times 1} \\
[?]_{1 \times (n-1)} & 0
\end{bmatrix} \quad \text{with $P(M) \in \Mata_{n-1}(\K)$.}$$
Assume that $P(S) \subset \WA_{n-1,1,1}(\K)$ and $\dim P(S)>1$.
Then, $S$ is congruent to a subspace of $\WA_{n,1,1}(\K)$.
\end{lemme}

\begin{proof}
For all $M \in S$, we split
$$M=\begin{bmatrix}
0 & -B(M)^T & ? \\
B(M) & [0]_{(n-2) \times (n-2)} & C(M) \\
? & -C(M)^T & 0
\end{bmatrix}
\quad \text{with $B(M) \in \K^{n-2}$ and $C(M) \in \K^{n-2}$.}$$
Setting
$$\calT:=\Bigl\{
\begin{bmatrix}
B(M) & C(M)
\end{bmatrix} \mid M \in S\Bigr\},$$
we proceed as in the proof of Proposition \ref{lifting2alt} to obtain
$$\urk \calT \leq 1.$$

Yet, we note that $\dim B(S)=\dim P(S)>1$, and hence no $1$-dimensional linear subspace of $\K^{n-2}$
includes the range of every matrix of $\calT$.
By the classification of spaces of matrices with rank at most 1, we deduce that some non-zero vector of $\K^2$ annihilates all the matrices in $\calT$.
As $B(S) \neq \{0\}$, such a vector cannot belong to $\K \times \{0\}$. Hence,
with a well-chosen congruence transformation we can reduce the situation to the one where $C=0$, in which case $S \subset \WA_{n,1,1}(\K)$.
\end{proof}

\begin{lemme}\label{lifting2symrank2}
Let $n$ be an integer such that $n\geq 3$.
Let $S$ be a linear subspace of $\Mats_n(\K)$ such that $\urk S \leq 2$.
Let us write every matrix $M$ of $S$ as
$$M=\begin{bmatrix}
P(M) & [?]_{(n-1) \times 1} \\
[?]_{1 \times (n-1)} & ?
\end{bmatrix} \quad \text{with $P(M) \in \Mats_{n-1}(\K)$.}$$
Assume that $P(S) \subset \WS_{n-1,1,0}(\K)$ and $\dim P(S)>2$.
Then, $S$ is congruent to a subspace of $\WS_{n,1,0}(\K)$.
\end{lemme}

\begin{proof}
For all $M \in S$, we split
$$M=\begin{bmatrix}
? & -B(M)^T & ? \\
B(M) & [0]_{(n-2) \times (n-2)} & C(M) \\
? & -C(M)^T & a(M)
\end{bmatrix}
\quad \text{with $B(M) \in \K^{n-2}$, $C(M) \in \K^{n-2}$ and $a(M) \in \K$.}$$
Setting
$$T:=\Bigl\{
\begin{bmatrix}
B(M) & C(M)
\end{bmatrix} \mid M \in S\Bigr\},$$
we proceed as in the proof of Proposition \ref{lifting2symeven} to obtain
$$\urk \calT \leq 1.$$
Moreover, we must have
$$\dim B(S) \geq \dim P(S)-1>1.$$
Then, with the same line of reasoning as in the proof of Lemma \ref{lifting2altrank2}, we obtain
that, after a well-chosen congruence transformation, no generality is lost in assuming that $C=0$.
Then, for all $M \in S$, we see that $\rk M \geq 3$ if $a(M) \neq 0$ and $B(M) \neq 0$.
If $a \neq 0$ then $\calS':=a^{-1} \{1\}$ is an affine hyperplane of $S$ on the whole of which $B$ vanishes;
yet $\dim B(\calS') \geq \dim B(S)-1>0$, yielding a contradiction. Hence, $a=0$ and $S \subset \WS_{n,1,0}(\K)$.
\end{proof}

\begin{lemme}\label{lifting2symrank3}
Let $n$ be an integer such that $n\geq 4$.
Let $S$ be a linear subspace of $\Mats_n(\K)$ such that $\urk S \leq 3$.
Let us write every matrix $M$ of $S$ as
$$M=\begin{bmatrix}
P(M) & [?]_{(n-1) \times 1} \\
[?]_{1 \times (n-1)} & ?
\end{bmatrix} \quad \text{with $P(M) \in \Mats_{n-1}(\K)$.}$$
Assume that $P(S) \subset \WS_{n-1,1,1}(\K)$ and $\dim P(S)>5$.
Then, $S$ is congruent to a subspace of $\WS_{n,1,1}(\K)$.
\end{lemme}

\begin{proof}
For all $M \in S$, we split
$$M=\begin{bmatrix}
? & ? & B(M)^T & ? \\
? & a(M) & [0]_{1 \times (n-3)} & b(M) \\
B(M) & [0]_{(n-3) \times 1} & [0]_{(n-3) \times (n-3)} & C(M) \\
? & b(M) & C(M)^T & c(M)
\end{bmatrix}$$
with $B(M) \in \K^{n-3}$, $C(M) \in \K^{n-3}$, and scalars $a(M)$, $b(M)$ and $c(M)$.
Once more, $T:=\Bigl\{
\begin{bmatrix}
B(M) & C(M)
\end{bmatrix} \mid M \in S\Bigr\}$
has upper-rank less than or equal to $1$.
Yet,
$$\dim B(S) \geq \dim P(S)-3>2.$$
With the same line of reasoning as in the previous lemmas, we see that no generality is lost in assuming that $C=0$.
Next, for $M \in S$, set
$$J(M):=\begin{bmatrix}
a(M) & b(M) \\
b(M) & c(M)
\end{bmatrix} \in \Mats_2(\K).$$
Once more, we see that
$$\forall M \in S, \; B(M) \neq 0 \Rightarrow \rk J(M) \leq 1.$$
Let us choose $M_1 \in S$ such that $B(M_1) \neq 0$, and then an index $i$ such that $B(M_1)_i \neq 0$.
We consider the affine hyperplane $\calS':=\{M \in S : \; B(M)_i=B(M_1)_i\}$. Then,
$J(\calS')$ is an affine subspace of $J(S)$ whose span equals $J(S)$, and every matrix in $J(\calS')$ has rank at most $1$.
Then, we know by Theorem 1.4 of \cite{dSPsym} that $\dim J(\calS') \leq 1$, and hence $\dim J(S) \leq 2$.
Assume now that $J(S)$ contains a rank $2$ matrix $F$.
Then, $\calS'':=\{M \in S : \; J(M)=F\}$ is an affine subspace of $S$
with codimension at most $2$, and $B$ vanishes everywhere on it. Yet $\dim B(\calS'') \geq \dim B(S)-2>0$,
leading to a contradiction.
Thus, $\urk J(S)\leq 1$, and hence, by Theorem 1.4 of \cite{dSPsym}, $J(S)$ is congruent to a subspace of $\Vect(E_{1,1})$.
We conclude that $S$ is congruent to a subspace of $\WS_{n,1,1}(\K)$.
\end{proof}

\subsection{Results on spaces with small upper-rank}

This very short section consists in the proofs of Propositions \ref{altrank2}, \ref{symrank2} and \ref{symrank3}.

We start with the proof of Proposition \ref{altrank2}.

The case $n \leq 3$ is vacuous.
Let $n>3$, and let $S$ be a linear subspace of $\Mata_n(\K)$ with $\urk S \leq 2$ and $\dim S>3$.

By Proposition \ref{keylemmaalternate}, $S$ is congruent to $\WA_{n,1,1}(\K)$
or there exists a linear hyperplane $H$ of $\K^n$
such that $S_H=\{0\}$. Assume that the second option holds.
Without loss of generality, we can assume that $H=\K^{n-1} \times \{0\}$.
Then, we split every matrix $M$ of $S$ up as
$$M=\begin{bmatrix}
P(M) & [?]_{(n-1) \times 1} \\
[?]_{1 \times (n-1)} & 0
\end{bmatrix} \quad \text{with $P(M) \in \Mata_{n-1}(\K)$.}$$
Obviously $\urk P(S) \leq 2$. On the other hand $\dim P(S)=\dim S>3$ since $S_H=\{0\}$.
Hence, by induction $P(S)$ is congruent to a subspace of $\WA_{n-1,1,1}(\K)$.
Then, without loss of generality we can assume that $P(S) \subset \WA_{n-1,1,1}(\K)$,
and Lemma \ref{lifting2altrank2} yields that $S$ is congruent to a subspace of $\WA_{n,1,1}(\K)$,
which completes the proof.

The proof of Proposition \ref{symrank2} is essentially similar, this time by using Proposition \ref{keylemmasymmetriceven}
and Lemma \ref{lifting2symrank2}.

Similarly, the proof of Proposition \ref{symrank3} is done by induction, with the help of Proposition \ref{keylemmasymmetricodd}
and Lemma \ref{lifting2symrank3}; indeed, the case $n=4$ is known to be vacuous by Theorem \ref{oldtheosym}.

\section{Proof of Theorem \ref{maintheoalt}}\label{altproofsection}

This section is devoted to the proof of Theorem \ref{maintheoalt}.
We shall use an induction on $n$, using similar techniques as in \cite{dSPsym}.

Let $n$ and $s$ be positive integers such that $2s<n$.

Let $S$ be a linear subspace of $\Mata_n(\K)$ such that
$$\urk S \leq 2s \quad \text{and} \quad \dim S>\max(a_{n,s-1,3},a_{n,1,2s-1}).$$
We wish to prove that $S$ is congruent to a subspace of $\WA_{n,s,1}(\K)$ or of $\WA_{n,0,2s+1}(\K)$.
If $2s=n-1$ then we directly have $S \subset \Mata_n(\K)=\WA_{n,0,2s+1}(\K)$.
In the rest of the proof, we assume that $2s<n-1$.
If $s=1$ we would have
$$\dim S>\max(a_{n,0,3},a_{n,1,1}),$$
contradicting Theorem \ref{oldtheoalt}. Therefore, $s \geq 2$.

Throughout the proof, we assume that $S$ is not congruent to a subspace of $\WA_{n,s,1}(\K)$.
We denote by $m$ the minimal dimension of $S_H$ when $H$ ranges over the linear hyperplanes
of $\K^n$. By Lemma \ref{keylemmaalternate}, we have
$$m\leq s-1.$$
Without loss of generality, we can assume that for $H:=\K^{n-1} \times \{0\}$ we have
$$m=\dim S_H.$$
Throughout the proof, we shall split every matrix $M \in S$ along the following pattern:
$$M=\begin{bmatrix}
P(M) & [?]_{(n-1) \times 1} \\
[?]_{1 \times (n-1)} & 0
\end{bmatrix} \quad \text{with $P(M) \in \Mata_{n-1}(\K)$.}$$
Note that
$$\dim P(S)=\dim S-\dim S_H = \dim S-m.$$
In particular, noting that
$$a_{n-1,s,1} = a_{n,s-1,3}+(n-s-3)-s$$
and that $m \leq s$, we deduce from $\dim S>a_{n,s-1,3}$ that
$$\dim P(S) > a_{n-1,s,1}-(n-s-3).$$
By Proposition \ref{lifting2alt}, if $P(S)$ were congruent to a subspace of
$\WA_{n-1,s,1}(\K)$, then $S$ would be congruent to a subspace of $\WA_{n,s,1}(\K)$,
contradicting one of our first assumptions.
Therefore:
\begin{center}
$P(S)$ is not congruent to a subspace of $\WA_{n-1,s,1}(\K)$.
\end{center}
From there, we split the discussion into two main subcases, whether $m=0$ or $m>0$.

\subsection{Case 1: $m=0$}

As $S_H=\{0\}$ we have
$$\dim P(S)=\dim S.$$
Obviously, $\urk P(S) \leq 2s$.

Since $a_{n,s-1,2} \geq a_{n-1,s-1,2}$ and $a_{n,1,2s-1} \geq a_{n-1,1,2s-1}$, we know by induction that
$P(S)$ is congruent to a subspace of $\WA_{n-1,s,1}(\K)$ or to a subspace of $\WA_{n-1,0,2s+1}(\K)$.
Hence, $P(S)$ is actually congruent to a subspace of $\WA_{n-1,0,2s+1}(\K)$,
and without loss of generality we can assume that $P(S) \subset \WA_{n-1,0,2s+1}(\K)$.

Since $n \geq 2s+2$, we find
$$\dim P(S)>a_{n,1,2s-1} = \dbinom{2s-1}{2}+(n-1) \geq \dbinom{2s}{2}+2.$$
Hence, Proposition \ref{lifting1alt} shows that $S$ is congruent to a subspace of
$\WA_{n,0,2s+1}(\K)$.

\subsection{Case 2: $m>0$}

Remember that $s \geq 2$.
Without loss of generality, we can assume that $S_H$ contains $E_{1,n}-E_{n,1}$.
We split every matrix $M$ of $S$ as
$$M=\begin{bmatrix}
0 & [?]_{1 \times (n-2)} & ? \\
[?]_{(n-2) \times 1} & K(M) & [?]_{(n-2) \times 1} \\
? & [?]_{1 \times (n-2)} & 0
\end{bmatrix}$$
with $K(M) \in \Mata_{n-2}(\K)$.
By the extraction lemma (Lemma \ref{extractioncoralt}), we find that
$$\urk K(M) \leq 2s-2.$$
On the other hand, by the rank theorem we find that
$$\dim K(S) \geq \dim S-(n-2)-\dim S_H = \dim S-(n-2)-m.$$

\begin{claim}
If $s>2$ then $\dim K(S)>\max\bigl(a_{n-2,s-2,3},a_{n-2,1,2s-3}\bigr)$.
\end{claim}

\begin{proof}
Assume that $s>2$.
One checks that
$$a_{n,s-1,3}=a_{n-2,s-2,3}+(n-2)+(s-1),$$
and hence $\dim K(S)>a_{n-2,s-2,3}$ follows from the assumption that $\dim S >a_{n,s-1,3}$ and that
$m \leq s-1$.

In the rest of the proof, we assume that $\dim K(S) \leq a_{n-2,1,2s-3}$ and we show that it leads to a contradiction.
First of all, we must have
\begin{equation}\label{ineq1}
a_{n-2,s-2,3} < a_{n-2,1,2s-3}.
\end{equation}

On the other hand, since $\dim S>a_{n,1,2s-1}$ we must have
\begin{equation}\label{ineq2}
a_{n-2,1,2s-3} > a_{n,1,2s-1}-(n-2)-m \geq a_{n,1,2s-1}-(n-2)-(s-1).
\end{equation}
Now, we prove that \eqref{ineq1} and \eqref{ineq2} are contradictory.
Indeed, by a straightforward computation, one checks that \eqref{ineq1} leads to
$$n(s-3) \leq \frac{5}{2}\, s^2-\frac{13}{2}\,s-4,$$
whereas \eqref{ineq2} leads to
$$3s+1 \leq n.$$
Since $s \geq 3$, combining those two inequalities
leads to
$$(3s+1)(s-3) \leq \frac{5}{2}\, s^2-\frac{13}{2}\,s-4,$$
which further leads to $\frac{s^2-3s}{2} \leq -1$ and contradicts the fact that $s \geq 3$.
\end{proof}

\begin{claim}
If $s=2$ then $\dim K(S)>3$.
\end{claim}

\begin{proof}
Assume that $s=2$. Then, $\dim S>a_{n,1,2s-1}$ reads $\dim S>(n-1)+3$.
Hence,
$$\dim K(S) \geq \dim S-(n-2)-m \geq \dim S-(n-1)>3.$$
\end{proof}

Combining the above two claims, we obtain (by induction if $s>2$, and by Proposition \ref{altrank2} otherwise)
that $K(S)$ is congruent to a subspace of $\WA_{n-2,s-1,1}(\K)$ or to a subspace of $\WA_{n-2,0,2s-1}(\K)$.
However, the first case would lead to $P(S)$ being congruent to a subspace of $\WA_{n-1,s,1}(\K)$, which has been ruled
out from the start.

It follows that no generality is lost in assuming that $K(S) \subset \WA_{n-2,0,2s-1}(\K)$.
We shall prove that this leads to a contradiction.
Setting $H':=\K^{n-2} \times \{0\} \times \K$, we see that $S_{H'} \subset \Vect(E_{1,n-1}-E_{n-1,1},E_{n,n-1}-E_{n-1,n})$,
which leads to $m \leq 2$.
Then,
\begin{multline*}
a_{n,1,2s-1}\leq \dim S -1\leq (n-2)+m+\dim K(S)-1 \\
\leq (n-2)+m+a_{n-2,0,2s-1}-1 \leq (n-1)+a_{n-2,0,2s-1}=a_{n,1,2s-1}.
\end{multline*}
Hence, all the intermediate inequalities turn out to be equalities, which yields:
\begin{enumerate}[(a)]
\item $m=2$;
\item $K(S)=\WA_{n-2,0,2s-1}(\K)$;
\item $P(S)$ contains $E_{1,j}-E_{j,1}$ for all $j \in \lcro 2,n-1\rcro$.
\end{enumerate}
In particular, as $m=2$ and $S_{H'} \subset \Vect(E_{1,n-1}-E_{n-1,1},E_{n,n-1}-E_{n-1,n})$
we gather that $S$ contains $E_{n-1,n}-E_{n,n-1}$.
On the other hand, by using properties (b) and (c) above, we get that for all
$A \in \Mata_{2s}(\K)$, the space $S$ contains a matrix of the form
$$\begin{bmatrix}
A & [?]_{2s \times (n-2s-1)} & [?]_{2s \times 1} \\
[?]_{(n-2s-1) \times 2s} & [0]_{(n-2s-1) \times (n-2s-1)} & [?]_{(n-2s-1) \times 1} \\
[?]_{1 \times 2s} & [?]_{1 \times (n-2s-1)} & 0
\end{bmatrix}.$$
Then, by the extraction lemma, we gather that $\rk A \leq 2s-2$ for all $A \in \Mata_{2s}(\K)$,
which is absurd.

This completes the proof of Theorem \ref{maintheoalt}.

\section{Proof of Theorem \ref{maintheosym}}\label{symproofsection}

This section is devoted to the proof of Theorem \ref{maintheosym}.
The strategy is essentially similar to the one of the proof of Theorem \ref{maintheoalt},
but with additional complexity due to the provision $r<n-1$ in Propositions \ref{keylemmasymmetriceven} and \ref{keylemmasymmetricodd} and
to the need of distinguishing between the even case and the odd case.

We proceed by induction on $n$. Throughout the section, we assume that $\K$ has more than $2$ elements.

Let $n,r$ be positive integers such that $n>r$, and set
$$s:=\left \lfloor \frac{r}{2}\right \rfloor \quad \text{and} \quad \epsilon:=r-2s.$$
Let $S$ be a linear subspace of $\Mats_n(\K)$ such that
$$\dim S>\max\left(s_{n,s-1,\epsilon+2}, s_{n,1,r-2}\right) \quad \text{and} \quad \urk S \leq r.$$
In particular, by Theorem \ref{oldtheosym} we must have $s \geq 2$.

Assume for a moment that the following condition is satisfied:
\begin{itemize}
\item[(H1)]
$\K$ has characteristic $2$ and $S$ is included in $\Mata_n(\K)$.
\end{itemize}
Then, the upper-rank of $S$ would read $2s'$ for some integer $s'$ such that $s' \leq s$;
noting that $s_{n,s-1,\epsilon+2} \geq a_{n,s'-1,1}$ and $s_{n,1,r-2} \geq a_{n,1,2s'-1}$,
we would deduce from Theorem \ref{maintheoalt} that $S$ is congruent to a subspace of $\WA_{n,s',1}(\K)$
or to a subspace of $\WA_{n,0,2s'+1}(\K)$. In the first case, $S$ would be congruent to a
subspace of $\WS_{n,s,\epsilon}(\K)$,
and in the second one it would be congruent to a subspace of $\WA_{n,0,r+1}(\K)$.
Moreover, if $r$ were odd and $S$ were congruent to a subspace of $\WA_{n,0,r+1}(\K)$, then Theorem \ref{oldtheoalt} would yield
$$\dim S \leq \dbinom{r}{2},$$
which would contradict our assumptions because
$$s_{n,1,r-2}=\dbinom{r-1}{2}+n \geq \dbinom{r-1}{2}+(r-1) = \dbinom{r}{2}.$$
Hence, in any case, we would obtain one of the desired outcomes.

In the rest of the proof, we assume that condition (H1) does not hold.
We also make the following additional assumption:
\begin{itemize}
\item[(H2)] The space $S$ is not congruent to a subspace of $\WS_{n,s,\epsilon}(\K)$.
\end{itemize}

By Propositions \ref{keylemmasymmetricsubfull}, \ref{keylemmasymmetriceven} and \ref{keylemmasymmetricodd} we can
find an $S$-adapted linear hyperplane $H$ with minimal dimension $m$, and we know that
$$m \leq s-1 \quad \text{or} \quad \begin{cases}
r=n-1 \\
m \leq n-3.
\end{cases}$$

Without loss of generality, we can also assume that the hyperplane $H:=\K^{n-1} \times \{0\}$ is $S$-adapted and satisfies
$$m=\dim S_H.$$

Throughout the proof, we shall split every matrix $M \in S$ up as
$$M=\begin{bmatrix}
P(M) & [?]_{(n-1) \times 1} \\
[?]_{1 \times (n-1)} & ?
\end{bmatrix} \quad \text{with $P(M) \in \Mats_{n-1}(\K)$.}$$
Note that
$$\dim P(S)=\dim S-\dim S_H = \dim S-m.$$
Moreover, since $H$ is $S$-adapted some matrix of $P(S)$ has its attached quadratic form non-zero in the event when $\K$ has characteristic $2$.

Assume for the moment that $r<n-1$.
By noting that
$$s_{n-1,s,0}=s_{n,s,0}-s=s_{n,s-1,2}+(n-s-3)-(s-1),$$
we deduce from $\dim S>s_{n,s-1,2+\epsilon}$ and $m \leq s-1$ that
$$\dim P(S) > s_{n-1,s,0}-(n-s+3) \quad \text{if $r$ is even.}$$
On the other hand, by noting that
$$s_{n-1,s,1}=s_{n,s,1}-s=s_{n,s-1,3}+(n-s-5)-(s-1)$$
we obtain that
$$\dim P(S) > s_{n-1,s,1}-(n-s-5) \quad \text{if $r$ is odd.}$$
If $P(S)$ were congruent to a subspace of $\WS_{n-1,s,\epsilon}(\K)$,
then by one of Propositions \ref{lifting2symeven} or \ref{lifting2symodd} we would find that
$S$ is congruent to a subspace of $\WS_{n,s,\epsilon}(\K)$, thereby contradicting (H2).
It follows that:

\begin{center}
If $r<n-1$ then $P(S)$ is not congruent to a subspace of $\WS_{n-1,s,\epsilon}(\K)$.
\end{center}

From there, we split the discussion into three subcases, whether $m=0$, $m>0$ and $r<n-1$, or
$m>0$ and $r=n-1$.

\subsection{Case 1: $m=0$}

As $S_H=\{0\}$ we have
$$\dim P(S)=\dim S.$$
Obviously, $\urk P(S) \leq r$, and since $H$ is $S$-adapted
the space $P(S)$ is not included in $\Mata_{n-1}(\K)$.

Since $s_{n,s-1,2+\epsilon} \geq s_{n-1,s-1,2+\epsilon}$ and $s_{n,1,r-2} \geq s_{n-1,1,r-2}$, we know that
one of the following conditions holds:
\begin{enumerate}
\item[(i)] $P(S)$ is congruent to a subspace of $\WS_{n-1,s,\epsilon}(\K)$;
\item[or] ${}$
\item[(ii)] $P(S)$ is congruent to a subspace of $\WS_{n-1,0,r}(\K)$.
\end{enumerate}
Indeed, this is given by induction if $r<n-1$, otherwise (ii) is obviously true.

Let us discard the first option.
First of all, we already know that it cannot occur if $r<n-1$. Assume now that $r=n-1$ and that (i) holds.
Then, we would have
$$\dim S=\dim P(S) \leq s_{n-1,s,\epsilon}$$
and hence
$$s_{n,1,n-3} +1 \leq s_{n-1,s,\epsilon.}$$
One checks that this would lead to $s^2-3s+6 \leq 0$ if $\epsilon=0$ and $s^2-s+4 \leq 0$ otherwise,
which is false.

Therefore, without loss of generality, we can now assume that $P(S) \subset \WS_{n-1,0,r}(\K)$.
Since $n \geq r+1$, we have $\dim P(S)>s_{n,1,r-2} \geq \dbinom{r}{2}+2$. Hence,
Proposition \ref{lifting1sym} applies to $S$, and we conclude that
$S$ is congruent to a subspace of $\WS_{n,0,r}(\K)$.

\subsection{Case 2: $r<n-1$ and $m>0$}

Then, we know that $m \leq s-1$.
Remember that $S_H$ contains no rank $1$ matrix. Consider an arbitrary rank $2$ matrix $N$ in $S_H$,
with kernel denoted by $G$. Note that $G$ is a linear hyperplane of $H$.
Assume that $\K$ has characteristic $2$.
Since $\# \K>2$, if a quadratic form $q$ on $H$ vanishes everywhere
on $H \setminus G$ then it is zero (indeed, by taking a linear form $\varphi$ on $H$ with kernel $G$, we would see that $x \mapsto q(x)\varphi(x)$,
which is a homogeneous polynomial of degree $3$, would vanish everywhere on $H$, leading to $q=0$ since $\varphi \neq 0)$.
Since $H$ is $S$-adapted we deduce that there exists a matrix
$M \in S$ and a non-zero vector $X \in H \setminus G$ such that $X^T M X \neq 0$.
Thus, by applying a congruence transformation, we see that no generality is lost in assuming that
$X$ is the first vector of the standard basis and that $G=\{0\} \times \K^{n-2} \times \{0\}$.

Hence, in the rest of the proof, we assume the following:
\begin{itemize}
\item The space $S$ contains $E_{1,n}+E_{n,1}+a\,E_{n,n}$ for some $a \in \K$.

\item If $\K$ has characteristic $2$ then some matrix of $S$ has a non-zero entry at the $(1,1)$-spot.
\end{itemize}

Next, we further split every matrix $M$ of $S$ as
$$M=\begin{bmatrix}
? & [?]_{1 \times (n-2)} & ? \\
[?]_{(n-2) \times 1} & K(M) & [?]_{(n-2) \times 1} \\
? & [?]_{1 \times (n-2)} & ?
\end{bmatrix} \quad \text{with $K(M) \in \Mats_{n-2}(\K)$.}$$

As $\# \K>2$, the extraction lemma (Lemma \ref{extractioncorsym}) leads to
$$\urk K(S) \leq r-2.$$
On the other hand, by the rank theorem we find that
$$\dim K(S) \geq \dim S-(n-1)-\dim S_H = \dim S-(n-1)-m.$$

\begin{claim}\label{claim3}
If $s>2$ then $\dim K(S)>\max(s_{n-2,s-2,\epsilon+2},s_{n-2,1,r-4})$.
\end{claim}

\begin{proof}
Assume that $s>2$.
One checks that
$$s_{n,s-1,\epsilon+2}=s_{n-2,s-2,\epsilon+2}+(n-1)+(s-1),$$
and hence inequality $\dim K(S)>s_{n-2,s-2,\epsilon+2}$ follows from the assumption that $\dim S >s_{n,s-1,\epsilon+2}$ and that
$m \leq s-1$.

In the rest of the proof, we assume that $\dim K(S) \leq s_{n-2,1,r-4}$ and we show that it leads to a contradiction.
First of all, we must have
\begin{equation}\label{ineq3}
s_{n-2,s-2,\epsilon+2} < s_{n-2,1,r-4}.
\end{equation}

On the other hand, since $\dim S>s_{n,1,r-2}$ we must have
\begin{equation}\label{ineq4}
s_{n-2,1,r-4} > s_{n,1,r-2}-(n-1)-m \geq s_{n,1,r-2}-(n-1)-(s-1).
\end{equation}
Now, we prove that \eqref{ineq3} and \eqref{ineq4} are contradictory.
First of all, inequality \eqref{ineq4} leads to
$n \geq 2r-s$, and hence
$$n \geq 3s+2\epsilon.$$

On the other hand, inequality \eqref{ineq3} leads to
$$\begin{cases}
\frac{5}{2} s^2-\frac{15}{2} s-1 \geq n(s-3) & \text{if $r$ is even,} \\
\frac{5}{2} s^2-\frac{11}{2} s-7 \geq n(s-3) & \text{if $r$ is odd.}
\end{cases}$$
Since $s \geq 3$, combining these two sets of inequalities yields, in any case:
$$\frac{s(s-3)}{2} \leq -1.$$
This contradicts the assumption that $s \geq 3$.
\end{proof}

\begin{claim}
Either the space $K(S)$ is congruent to a subspace of $\WS_{n-2,0,r-2}(\K)$, or $\K$ has characteristic $2$, $r$ is even
and $K(S)$ is congruent to a subspace of $\WA_{n-2,0,r-1}(\K)$.
\end{claim}

\begin{proof}
If $r\in \{4,5\}$ then
$$\dim K(S) \geq \dim S-(n-1)-1 > s_{n,1,\epsilon+2}-n =\dbinom{\epsilon+3}{2},$$
and hence $K(S)$ satisfies the assumptions of Proposition \ref{symrank2} if $r=4$, and of Proposition \ref{symrank3}
if $r=5$.

On the other hand, if $r>5$ then Claim \ref{claim3} shows that we can apply the induction hypothesis.

Hence, $K(S)$ is congruent to a subspace of $\WS_{n-2,s-1,\epsilon}(\K)$
or the claimed outcome holds.
However, the first option must be discarded since $P(S)$ is not congruent to a subspace of $\WS_{n-1,s,\epsilon}(\K)$.
\end{proof}

Without loss of generality, we can now assume that $K(S) \subset \WS_{n-2,0,r-2}(\K)$ or
$\K$ has characteristic $2$, $r$ is even and $K(S) \subset \WA_{n-2,0,r-1}(\K)$.

In any case, by setting $H':=\K^{n-2} \times \{0\} \times \K$, we see that
$$S_{H'} \subset \Vect(E_{1,n-1}+E_{n-1,1},E_{n,n-1}+E_{n-1,n})$$
since $r-1<n-2$.

Remembering that some matrix of $S$ has a non-zero entry at the $(1,1)$-spot if $\K$ has characteristic $2$, we see that $H'$ is $S$-adapted, and
we deduce that $m \leq 2$.
Hence,
\begin{multline*}
s_{n,1,r-2}\leq \dim S -1\leq (n-1)+m+\dim K(S)-1 \leq (n-2)+m+s_{n-2,0,r-2} \\
\leq (n-2)+2+s_{n-2,0,r-2}=s_{n,1,r-2}.
\end{multline*}
Hence, all the intermediate inequalities turn out to be equalities, which yields:
\begin{enumerate}[(a)]
\item $m=2$;
\item Either $K(S)=\WS_{n-2,0,r-2}(\K)$, or $\K$ has characteristic $2$, $r$ is even and $K(S)=\WA_{n-2,0,r-1}(\K)$;
\item $P(S)$ contains $E_{1,j}+E_{j,1}$ for all $j \in \lcro 2,n-1\rcro$, and it contains $E_{1,1}$.
\end{enumerate}
In particular, as $m=2$ and $S_{H'} \subset \Vect(E_{1,n-1}+E_{n-1,1},E_{n,n-1}+E_{n-1,n})$
we gather that $S$ contains $E_{n-1,n}+E_{n,n-1}$.

Now, we distinguish between two cases. Assume first that $K(S)=\WS_{n-2,0,r-2}(\K)$.
Combining this with (c), we get that
the space $S$ contains a matrix of the form
$$\begin{bmatrix}
I_{r-1} & [?]_{(r-1) \times (n-r+1)} \\
[?]_{(n-r+1) \times (r-1)} & [?]_{(n-r+1) \times (n-r+1)}
\end{bmatrix}.$$
Then, by the extraction lemma (Corollary \ref{extractioncorsym}), we would obtain that $I_{r-1}$ has rank less than $r-1$, which is false.

Assume now that $\K$ has characteristic $2$, $r$ is even and $K(S)=\WA_{n-2,0,r-1}(\K)$.
Then, we obtain that $S$ contains a matrix of the form
$$\begin{bmatrix}
0_{s \times s} & I_s & [?]_{s \times (n-2s)} \\
I_s & 0_{s \times s} & [?]_{s \times (n-2s)} \\
[?]_{(n-2s) \times s} & [?]_{(n-2s) \times s} & [?]_{(n-2s) \times (n-2s)}
\end{bmatrix}.$$
As in the above, since $2s \leq n-2$ the extraction lemma would yield that $\begin{bmatrix}
0_{s \times s} & I_s  \\
I_s & 0_{s \times s}
\end{bmatrix}$ has rank less than $r-2$, which is false.

This completes the study of Case 2.

\subsection{Case 3: $r=n-1$ and $m>0$}

Then, we know that $m \leq n-3$. Once more, we shall prove that this leads to a contradiction.
The strategy is globally similar to the one in Case 2, with increased technicalities however.
Throughout the proof, it will be useful to note that
$$s_{n,0,n-1} \geq s_{n,s,\epsilon},$$
which is obtained by a straightforward computation.

We start with a simple result.

\begin{claim}
One has $m \geq 2$.
\end{claim}

\begin{proof}
Assume on the contrary that $m=1$. Without loss of generality, we can assume that $S_H$ contains $E_{1,n}+E_{n,1}+a E_{n,n}$
for some $a \in \K$. Then, with the same notation as in the previous section,
we would deduce from Theorem \ref{oldtheosym} that $\dim K(S) \leq \max(s_{n-2,s-1,\epsilon},s_{n-2,0,n-3})=s_{n-2,0,n-3}$.
Using the rank theorem once more, it would follow that
$$\dim S \leq (n-1)+m+\dim K(S) \leq n+s_{n-2,0,n-3}=s_{n,1,n-3},$$
contradicting our assumptions.
\end{proof}

In the rest of the proof, we shall need to consider the various block matrices in a more conceptual fashion:
we shall think in terms of symmetric bilinear forms.

\begin{Not}
Let $G$ be an arbitrary linear hyperplane of $\K^n$, and
let $N \in S_G$ be of rank $2$. We define
$$P_G(S):=\bigl\{(X,Y) \in G^2 \mapsto X^T M Y \mid M \in S\bigr\}$$
$$K_N(S):=\bigl\{ (X,Y) \in (\Ker N)^2 \mapsto X^T M Y \mid M \in S\bigr\},$$
which are a linear subspaces, respectively, of the space of all symmetric bilinear forms on $G$, and of the space of all symmetric
bilinear forms on $\Ker N$.
\end{Not}

To make things clearer, if $G=\K^{n-1} \times \{0\}$ and $N=E_{1,n}+E_{n,1}+a\,E_{n,n}$ for some $a \in \K$, then
with our usual notation $P(S)$ represents $P_G(S)$ in the standard basis of $G$,
and $K(S)$ represents $K_N(S)$ in the standard basis of $\Ker N=\{0\} \times \K^{n-2} \times \{0\}$.

Note that the extraction lemma reads as follows:

\begin{lemme}
Let $N$ be a rank $2$ matrix of $S$ that has a totally isotropic hyperplane.
Then, $\urk K_N(S) \leq r-2$.
\end{lemme}

This motivates the introduction of the following terminology:

\begin{Def}
A rank $2$ matrix $N \in \Mats_n(\K)$ is called \textbf{$S$-good} if
either $\K$ has characteristic not $2$ or some matrix $M \in S$ is not totally isotropic on $\Ker N$.
\end{Def}

\begin{claim}\label{existsSgoodclaim}
For every $S$-adapted linear hyperplane $G$ of $\K^n$, the space $S_G$ contains an $S$-good matrix.
\end{claim}

\begin{proof}
Let $G$ be an $S$-adapted linear hyperplane of $\K^n$.
Then, $\dim S_G \geq m \geq 2$. Consider two linearly independent matrices $N_1$ in $N_2$ in $S_G$.
Assume that none of them is $S$-good, so that $\K$ has characteristic $2$.
There exists a matrix $M \in S$ such that the quadratic form $q : X \in G \mapsto X^T MX$ is non-zero.
Since $\K$ has characteristic $2$ and $M$ is symmetric, the isotropy cone $q^{-1} \{0\}$
is actually a linear subspace of $G$. Since none of $N_1$ and $N_2$ is $S$-good, it follows that
$\Ker N_1 \subset q^{-1}\{0\}$ and $\Ker N_2 \subset q^{-1}\{0\}$.
On the other hand $S_G$ contains no rank $1$ matrix, whence no non-trivial linear combination of $N_1$ and $N_2$
has rank $1$. It follows that $\Ker N_1$ and $\Ker N_2$ are distinct hyperplanes of $G$. Since $q^{-1}\{0\}$ is a linear subspace of $G$ that includes
the distinct linear hyperplanes $\Ker N_1$ and $\Ker N_2$, we conclude that $q^{-1}\{0\}=G$, contradicting the fact that $q$ is non-zero.

Hence, one of the matrices $N_1$ and $N_2$ is $S$-good.
\end{proof}

Now, let $G$ be an arbitrary $S$-adapted linear hyperplane of $\K^n$ such that $\dim S_G=m$.
We can choose an $S$-good matrix $N$ in $S_G$.
The rank theorem yields
$$\dim K_N(S) \geq \dim S-(n-1)-\dim S_G = \dim S-(n-1)-m.$$

With the same line of reasoning as in the end of the study of Case 2, we obtain:

\begin{claim}\label{nottype1claim}
Let $N$ be a rank $2$ matrix in $S_G$ for some $S$-adapted linear hyperplane $G$ such that $\dim S_G=m$.
Then, $K_N(S)$ is not represented by a subspace of $\WS_{n-2,0,n-3}(\K)$.
\end{claim}

Finally, one checks that $s_{n-2,s-2,\epsilon+2} \leq s_{n-2,1,n-5}$ unless $s=2$ and $\epsilon=1$ (in which case $n=6$).
This motivates that we tackle the case $n=6$ separately.

\subsubsection{Subcase 3.1: $n=6$}

Here, the assumptions on the dimension of $S$ tell us that $\dim S \geq 13$.
Let us consider the $S$-adapted linear hyperplane $H:=\K^{n-1} \times \{0\}$. Without loss of generality,
we can assume that, for some list $(\alpha_1,\dots,\alpha_m) \in \K^m$, the space
$S_H$ contains the matrix $E_{n,i}+E_{i,n}+\alpha_i E_{n,n}$ for all $i \in \lcro 1,m\rcro$.

As $K_N(S)$ is not represented by a subspace of $\WS_{4,0,3}(\K)$, it follows from Theorem \ref{oldtheosym} that
$\dim K_N(S)<6$.
Yet, by the rank theorem
$$\dim K_N(S) \geq \dim S-(n-1)-m \geq 13-(n-1)-(n-3) = 5.$$
We deduce that $m=n-3=3$ and that $P(S)$ contains $E_{1,j}+E_{j,1}$ for all $j \in \lcro 2,5\rcro$.
By applying the same method to the matrices $E_{i,n}+E_{n,i}+\alpha_i E_{n,n}$, with $i$ in $\lcro 2,3\rcro$,
we find that $P(S)$ contains $E_{i,j}+E_{j,i}$ for all $(i,j)\in \lcro 1,5\rcro^2 \setminus
\{ 4,5\}^2$ such that $i \neq j$.
However, it would follow that $K(S)$ contains the rank $4$ matrix $E_{4,1}+E_{3,2}+E_{2,3}+E_{1,4}$,
contradicting the fact that $\urk K(S) \leq 3$.

\subsubsection{Subcase 3.2:  $n \neq 6$}

Let $G$ be an $S$-adapted linear hyperplane of $\K^n$.
For any $S$-good matrix $N$ in $S_G$, the rank theorem yields
\begin{equation}\label{majodimKSN}
\dim K_N(S) \geq \dim S-(n-1)-m \geq s_{n,1,n-3}+1-(n-1)-(n-3)=s_{n-2,1,n-5}.
\end{equation}

\begin{claim}\label{shittyclaim1}
Let $G$ be an $S$-adapted linear hyperplane of $\K^n$ such that $\dim S_G=m$.
Assume that there is an $S$-good matrix $N$ in $S_G$ such that $\dim K_N(S)>s_{n-2,1,n-5}$.
Then, $n=7$, $m=3$, and $P_G(S)$ is represented by $\WS_{6,3,0}(\K)$.
\end{claim}

\begin{proof}
Without loss of generality, we can assume that $G=\K^{n-1} \times \{0\}$ and $N=E_{1,n}+E_{n,1}+a\,E_{n,n}$
for some $a \in \K$. We use the same notation as before, so that $K(S)$ represents $K_N(S)$ and
$P(S)$ represents $P_G(S)$.

Since $n \neq 6$ we know that $s_{n-2,s-2,\epsilon+2} \leq s_{n-2,1,n-5}$,
and hence the induction hypothesis applies to $K(S)$.
Using Claim \ref{nottype1claim}, we deduce that
$K(S)$ is congruent to a subspace of $\WS_{n-2,s-1,\epsilon}(\K)$
(indeed, if $\K$ has characteristic $2$ then, $N$ being $S$-good, the space $K(S)$ must contain a non-alternating matrix).

We deduce that $s_{n-2,s-1,\epsilon} \geq \dim K_N(S) \geq s_{n-2,1,n-5}+1$. A straightforward computation shows that
this yields $\epsilon=0$, $s \in \{3,4\}$, and $s_{n-2,s-1,\epsilon}=s_{n-2,1,n-5}+1$. In turn, the latter equality shows that
$K(S)$ is congruent to $\WS_{n-2,s-1,0}(\K)$.
Coming back to \eqref{majodimKSN}, we deduce that:
\begin{itemize}
\item $m \geq n-4$;
\item If $m=n-4$ then $P(S)$ contains $E_{1,1}$ and it contains $E_{1,i}+E_{i,1}$ for all $i \in \lcro 2,n-1\rcro$.
\end{itemize}

We shall prove that $m \leq s$. Assume on the contrary that $m \geq s+1$.
Then, we perform an additional congruence transformation to reduce the situation to the one
where $K(S)=\WS_{n-2,s-1,0}(\K)$. Hence, $P(S) \subset \WS_{n-1,s,0}(\K)$.
Let $j \in \lcro s+1,n-1\rcro$ and consider the linear hyperplane $H_j$ of $\K^n$ defined by the equation $x_j=0$ in the standard basis.
Note that $H_j$ is $S$-adapted: indeed, on the one hand the inclusion $P(S) \subset \WS_{n-1,s,0}(\K)$ yields that $S_{H_j}$ contains no rank $1$ matrix;
on the other hand if $\K$ has characteristic $2$, then as $G$ is $S$-adapted some matrix of $P(S)$ has a non-zero diagonal entry,
and since $P(S) \subset \WS_{n-1,s,0}(\K)$ the row index of such an entry must belong to $\lcro 1,s\rcro$, yielding
a non-alternating form in $P_{H_j}(S)$. Moreover, $S_{H_j} \subset \Vect(E_{i,j}+E_{j,i})_{i \in \lcro 1,s\rcro \cup \{n\}}$.
Since $\dim S_{H_j} \geq m \geq s+1$, we deduce that $S_{H_j}= \Vect(E_{i,j}+E_{j,i})_{i \in \lcro 1,s\rcro \cup \{n\}}$.
Hence, $S$ contains
$E_{i,j}+E_{j,i}$ for all $i \in \lcro 1,s\rcro \cup \{n\}$ and all $j \in \lcro s+1,n-1\rcro$.
In turn, this shows that $S$ is congruent to a linear subspace $\calT$ of $\Mats_n(\K)$ which contains
$E_{i,j}+E_{j,i}$ for all $i \in \lcro 1,s\rcro$ and all $j \in \lcro s+1,n\rcro$ (remember that $n$ is odd, whence $n-s=s+1$).
From there, with the line of reasoning from the end of the proof of Proposition \ref{keylemmasymmetriceven},
one can use the extraction lemma to prove that $\calT \subset \WS_{n,s,0}(\K)$, contradicting assumption (H2).

It follows that $m \leq s$. Remembering that $n$ is odd and that $n-4 \leq m$, this yields $2s-3 \leq s$, whence $s \leq 3$.
In turn, this shows that $s=3$, $n=7$ and $m=s$. Hence, $m=n-4$, and we deduce that $P(S)$ contains $E_{1,1}$ and $E_{1,i}+E_{i,1}$ for all $i \in \lcro 2,6\rcro$.
Since $K(S)$ is congruent to $\WS_{5,2,0}(\K)$, we conclude that $P(S)$ is congruent to $\WS_{6,3,0}(\K)$.
\end{proof}

\begin{claim}\label{shittyclaim2}
Let $G$ be an $S$-adapted linear hyperplane of $\K^n$ such that $\dim S_G=m$.
Assume that there is an $S$-good matrix $N$ in $S_G$ such that $\dim K_N(S)=s_{n-2,1,n-5}$.
Then, $n=5$, $m=2$, and $P_G(S)$ is represented by $\WS_{4,2,0}(\K)$.
\end{claim}

\begin{proof}
To simplify things, we can assume that $G=\K^{n-1} \times \{0\}$, $N=E_{1,n}+E_{n,1}+a\,E_{n,n}$
for some $a \in \K$, and $G$ contains a matrix of the form $E_{i,n}+E_{n,i}+ ? E_{n,n}$ for all $i \in \lcro 2,m\rcro$.

We use the same notation as before, so that $K(S)$ represents $K_N(S)$ and
$P(S)$ represents $P_G(S)$.
Coming back to \eqref{majodimKSN}, we obtain:
\begin{itemize}
\item $m=n-3$;
\item The space $P_G(S)$ contains $E_{1,1}$, and it contains $E_{1,j}+E_{j,1}$ for all $j \in \lcro 2,n-1\rcro$.
\end{itemize}
Since $P_G(S)$ contains $E_{1,1}$ and $S_G$ contains no rank $1$ matrix,  every matrix of $S_G$ that is not collinear to $N$
is $S$-good. Since $m=n-3$, Claim \ref{shittyclaim1} combined with inequality \eqref{majodimKSN}
shows that $\dim K_{N'}(S)=s_{n-2,1,n-5}$ for \emph{every} non-zero matrix $N'$ in $S_G$.
Hence, with the same line of reasoning applied to the $E_{i,n}+E_{n,i}+ ? E_{n,n}$ matrices,
we obtain that $P(S)$ contains $E_{i,i}$ for all $i \in \lcro 1,n-3\rcro$, and it contains $E_{i,j}+E_{j,i}$
for all distinct $i$ and $j$ that do not both belong to $\{n-2,n-1\}^2$.
In turn, this shows that $K(S)$ contains $E_{i,j}+E_{j,i}$ for all $(i,j)\in \lcro 1,n-2\rcro^2 \setminus \{n-3,n-2\}^2$ with $i \neq j$, and it contains
$E_{i,i}$ for all $i \in \lcro 1,n-4\rcro$. If $2(n-4) \geq n-2$, it is obvious that some linear combination of those matrices
is invertible, contradicting $\urk K(S)<n-2$. Hence, $2(n-4) \leq n-3$, which leads to $n \leq 5$. Yet, $s \geq 2$,
and hence $n=5$. Then, we have just shown that $\WS_{3,1,0}(\K) \subset K(S)$. However,
$\dim K(S) \leq 3$ by Theorem \ref{maintheosym}, whence $K(S)=\WS_{3,1,0}(\K)$.
Since $P(S)$ contains $E_{1,1}$ and $E_{1,j}+E_{j,1}$ for all $j \in \lcro 2,4\rcro$, we conclude that $P(S)=\WS_{4,2,0}(\K)$.
This proves the claimed result.
\end{proof}

Now, we are close to the conclusion of our proof. Using the above two claims and Claim \ref{existsSgoodclaim}, one sees that either $m=n-3$ or $m=n-4$.
In the first case, the assumptions of Claim \ref{shittyclaim1} cannot hold, and in the second one
the assumptions of Claim \ref{shittyclaim2} cannot hold.
Hence, we obtain the following results:

\begin{itemize}
\item[(H3)] The integer $n$ is odd;
\item[(H4)] One has $m=s$;
\item[(H5)] For every $S$-adapted linear hyperplane $G$ of $\K^n$
such that $\dim S_G=m$, the space $P_G(S)$ is represented by $\WS_{n-1,s,0}(\K)$.
\end{itemize}
Actually, we even have $n \in \{5,7\}$ but we will not use this fact in the remainder of the proof.

Without loss of generality, we can assume that $H=\K^{n-1} \times \{0\}$ is $S$-adapted with $\dim S_H=m$
and that $P(S)=\WS_{n-1,s,0}(\K)$.
Let us write every matrix $M$ of $S$ as
$$M=\begin{bmatrix}
[?]_{s \times s} & B(M)^T & [?]_{s \times 1} \\
B(M) & [0]_{(n-s-1) \times (n-s-1)} & C(M) \\
[?]_{1 \times s} & C(M)^T & a(M)
\end{bmatrix}$$
with $B(M) \in \Mat_s(\K)$, $C(M) \in \K^s$ and $a(M) \in \K$.

Set
$$\calT:=\Bigl\{\begin{bmatrix}
B(M) & C(M)
\end{bmatrix} \mid M \in S\Bigr\}$$
and note that $B(S)=\Mat_s(\K)$ since $P(S)=\WS_{n-1,s,0}(\K)$.
Our aim is to reduce the situation to the one where $C=0$.
This involves two steps.

\begin{claim}
For all $Z \in \K^s \setminus \{0\}$, we have
$\dim(\calT^T Z)=s$.
\end{claim}

\begin{proof}
Without loss of generality, we can assume that $Z$ is the first vector of the standard basis.
Since $B(S)=\Mat_s(\K)$, it is obvious that $\dim \calT^T Z \geq s$.

Denote by $e_1,\dots,e_n$ the standard basis of $\K^n$.
Let us consider the hyperplane $H':=\K^{n-2} \times \{0\} \times \K$.
Note that $S_{H'}$ is included in $\Vect(E_{i,n-1}+E_{n-1,i})_{1 \leq i \leq s} +\Vect(E_{n,n-1}+E_{n-1,n})$, which has dimension $s+1$.
If $S$ contained $E_{n-1,n}+E_{n,n-1}$, then the extraction lemma would yield that every matrix of
$\WS_{n-2,s,0}(\K)$ is singular, which is obviously false.
Hence, $\dim S_{H'} \leq s$. On the other hand, since $P(S)$ contains $E_{1,1}$ we see that
some bilinear form in $P_{H'}(S)$ is non-alternating. Finally, no matrix in $S_{H'}$ has rank $1$.
It follows that $H'$ is $S$-adapted with the minimal dimension.
Using (H5), we obtain an $s$-dimensional linear subspace $V$ of $H'$ such that
$P_{H'}(S)$ is the space of all bilinear forms on $H'$ that are totally singular on $V$.
Yet, all the bilinear forms in $P_{H'}(S)$ are already totally singular on $V_0 \cap H'$,
where $V_0:=\{0_s\} \times \K^s \times \{0\}$. Hence, $V_0 \cap H' \subset V$, and in particular $e_{s+1} \in V$.
Hence, for all $M$ in $S$, we have $X^T Me_{s+1}=0$ for all $X \in V$, as well as for all $X \in V_0$.
It follows that $S e_{s+1}$ is included in the orthogonal complement of $V+V_0$ for the standard symmetric bilinear form on $\K^n$.
However $V_0$ is not included in $V$
(since it is not included in $H'$), whence $V \neq V_0$. Since $\dim V=\dim V_0=s$, we deduce that $\dim(V+V_0)\geq s+1$, and we conclude that
$\dim \calT^T Z \leq s$, and we conclude that $\dim \calT^T Z=s$.
\end{proof}

\begin{claim}
There is a non-zero vector $X \in \K^{s+1}$ such that $NX=0$ for all $N \in \calT$.
\end{claim}

\begin{proof}
Denote by $(e_1,\dots,e_s)$ the standard basis of $\K^s$.
The linear map
$$\varphi : N \in \calT^T \mapsto (Ne_1,\dots, Ne_s) \in \underset{i=1}{\overset{s}{\prod}} \calT^T e_i$$
is injective, whereas $\dim \calT^T \geq \dim B(S)=s^2 =\dim \underset{i=1}{\overset{s}{\prod}} \calT^T e_i$.
Hence, $\varphi$ is a vector space isomorphism. It follows that, for $x:=\underset{i=1}{\overset{s}{\sum}} e_i$, we have
$\calT^T x=\underset{i=1}{\overset{s}{\sum}} \calT^T e_i$, and as $\dim \calT^T x=s$ we conclude that $\calT^T e_i=\calT^T x$ for all $i \in \lcro 1,s\rcro$.
Since $\calT^T x$ is a linear hyperplane of $\K^{s+1}$, this yields a non-zero vector $X \in \K^{s+1}$ such that $X^T N=0$ for all $N \in \calT^T$,
and the claimed result follows.
\end{proof}

Finally, the last entry of $X$ is non-zero since $B(S)=\Mat_s(\K)$. Hence, no generality is lost in assuming that
$X=\begin{bmatrix}
Y \\
-1
\end{bmatrix}$ for some $Y \in \K^s$, in which case $C(M)=B(M)Y$ for all $M \in S$.
Therefore, with the help of an additional harmless congruence transformation,
we can now assume that $C(M)=0$ for all $M \in S$.
Then, for all $M \in S$, we have $\rk B(M)<s$ whenever $a(M)\neq 0$.
If the linear form $a$ were non-zero, then $\calS':=\{M \in S : \; a(M)=1\}$ would be an affine hyperplane of
$S$, and $B(\calS')$ would be an affine subspace of $\Mat_s(\K)$ with codimension at most $1$ in which all the matrices are singular,
contradicting the affine version of Flanders's theorem (see \cite{dSPaffpres}).
Hence, $a(M)=0$ for all $M \in S$, and we conclude that $S \subset \WS_{n,s,0}(\K)$
(this actually contradicts assumption (H2), but never mind).

Our proof of Theorem \ref{maintheosym} is now complete.


\begin{thebibliography}{10}
\bibitem{AtkLloyd}
M.D. Atkinson and S. Lloyd,
{Large spaces of matrices of bounded rank.}
Quart. J. Math. Oxford (2)
{\bf 31} (1980) 253--262.

\bibitem{Beasley}
L.B. Beasley,
{Null spaces of spaces of matrices of bounded rank.}
Current Trends in Matrix Theory, Elsevier, 1987, 45--50.

\bibitem{ChooLimNg}
W.L. Chooi, M.H. Lim and Z.C. Ng,
{Linear spaces and preservers of symmetric matrices of bounded rank-two.}
Linear Multilinear Algebra
{\bf 61} (2013) 1051--1062.

\bibitem{Flanders}
H. Flanders,
{On spaces of linear transformations with bounded rank.}
J. Lond. Math. Soc.
{\bf 37} (1962) 10--16.

\bibitem{Lim}
M.H. Lim,
{Linear transformations on symmetric matrices.}
Linear Multilinear Algebra
{\bf 7} (1979) 47-–57.

\bibitem{Loewy}
R. Loewy,
{Large spaces of symmetric matrices of bounded rank are decomposable.}
Linear Multilinear Algebra
{\bf 48} (2001) 355--382.

\bibitem{LoewyRadwan}
R. Loewy and N. Radwan,
{Spaces of symmetric matrices of bounded rank.}
Linear Algebra Appl.
{\bf 197-198} (1994) 189--215.

\bibitem{Meshulamsymmetric}
R. Meshulam,
{On two extremal matrix problems.}
Linear Algebra Appl.
{\bf 114-115} (1989) 261--271.

\bibitem{dSPsym}
C. de Seguins Pazzis,
{Affine spaces of symmetric or alternating matrices with bounded rank.}
Linear Algebra Appl.
{\bf 504} (2016) 503–-558.

\bibitem{dSPboundedrankv2}
C. de Seguins Pazzis,
{Large spaces of matrices with bounded rank revisited.}
Linear Algebra Appl.
{\bf 504} (2016) 124--189.

\bibitem{dSPRCsym}
C. de Seguins Pazzis,
{Range-compatible homomorphisms on spaces of symmetric or alternating matrices.}
Linear Algebra Appl.
{\bf 503} (2016) 135--163.

\bibitem{dSPaffpres}
C. de Seguins Pazzis,
{The affine preservers of non-singular matrices.}
Arch. Math.
{\bf 95} (2010) 333--342.

\bibitem{dSPboundedrank}
C. de Seguins Pazzis,
{The classification of large spaces of matrices with bounded rank.}
Israel J. Math.
{\bf 208} (2015) 219--259.

\end{thebibliography}
\end{document}